\documentclass{amsart}

\usepackage{amsmath,amsthm,amssymb}
\usepackage{mathabx}
\usepackage{mathtools}
\usepackage{tikz}
\usepackage{standalone}
\usepackage{subcaption}
\usepackage{enumitem}
\usepackage{hyperref}

\usetikzlibrary{shapes.geometric}
\usetikzlibrary{calc}

\tikzstyle{BlueLine}=[line width=0.3mm,color=blue,text=black]
\tikzstyle{BluePoly}=[BlueLine,fill=blue!20]
\tikzstyle{RedLine}=[line width=0.3mm,color=red,text=black]
\tikzstyle{RedPoly}=[RedLine,fill=red!20]
\tikzstyle{GreenLine}=[thick,draw=black!30!green,text=black]
\tikzstyle{GreenPoly}=[thick,draw=green!50!black,fill=green!30,join=bevel]
\tikzstyle{OrangeLine}=[thick,color=orange]
\tikzstyle{GrayLine}=[thick,color=black!50!gray]
\tikzstyle{GrayPoly}=[GrayLine,fill=gray!20]
\tikzstyle{dot}=[shape=circle,draw,color=black,fill=black,inner sep=1.5pt]
\tikzstyle{bigdot}=[dot,inner sep=2pt]
\tikzstyle{littledot}=[dot,inner sep=1.2pt]
\tikzstyle{disk}=[thick,shape=circle,draw,color=black,fill=yellow!10]
\tikzstyle{plate}=[thick,shape=rectangle,draw,color=black,fill=yellow!10,
	rounded corners,minimum size=1.1cm]

\tikzstyle{dot}=[shape=circle,draw,color=black,fill=black,inner sep=1.5pt]
\tikzstyle{opendot}=[dot,fill=white]
\tikzstyle{disk}=[thick,shape=circle,draw,color=black]

\theoremstyle{plain}
\newtheorem{thm}{Theorem}[section]
\newtheorem{lem}[thm]{Lemma}
\newtheorem{cor}[thm]{Corollary}

\newtheorem{thma}{Theorem}

\theoremstyle{definition}
\newtheorem{defn}[thm]{Definition}
\newtheorem{rem}[thm]{Remark}

\newtheorem{example}[thm]{Example}

\newcommand{\bool}{\textsc{Bool}}
\newcommand{\ncht}{\textsc{NCHT}}

\newcommand{\upper}{{\uparrow}}
\newcommand{\low}{{\downarrow}}
\newcommand{\tri}{\textsc{Tri}}
\newcommand{\planar}{\textsc{Pnr}}
\newcommand{\planted}{\textsc{Ptd}}
\newcommand{\reduced}{\textsc{Red}}

\newcommand{\val}{{\operatorfont val}}
\newcommand{\PMod}{{\operatorfont PMod}}
\newcommand{\nbhd}{{\operatorfont nbhd}}
\newcommand{\pprime}{{\prime\prime}}

\newcommand{\ustar}{\Asterisk^\circ}
\newcommand{\mstar}{\Asterisk^\bullet}

\subjclass[2010]{}

\begin{document}
	
	\title{The Polyhedral Tree Complex}
	\author{Michael Dougherty} 
	\email{doughemj@lafayette.edu}
	\address{Department of Mathematics, 
		Lafayette College, Easton, PA 18042}
	
	\begin{abstract}
		The tree complex is a simplicial complex defined in recent
		work of Belk, Lanier, Margalit, and Winarski with 
		applications to mapping class groups and complex dynamics.
		This article introduces a connection between this setting and 
		the convex polytopes known as associahedra and cyclohedra. 
		Specifically, we describe a characterization of these polytopes using
		planar embeddings of trees and show that the tree complex is the barycentric 
		subdivision of a polyhedral cell complex for which the cells
		are products of associahedra and cyclohedra. 
	\end{abstract}
	
	\maketitle
	
	\section*{Introduction}
	
	Convex polytopes which arise from combinatorial structures
	form a rich area of study with an increasing number of applications.
	Associahedra, cyclohedra, and permutahedra are
	some of the classic examples \cite{bott94, simion03, ziegler95}, while more recent cases 
	include the the use of generalized associahedra in the study of cluster algebras \cite{fomin07}
	and the appearance of amplituhedra in work on scattering amplitudes \cite{arkani-hamed}.
	In each of these examples, the convex polytope can be combinatorially
	defined by describing the partially ordered set of its faces, i.e.
	the face poset. 
	The main goal of this article is to introduce a polytope
	arising from a partial order on a certain type of trees
	and examine an associated cell complex.
	
	A \emph{marked $n$-tree} (or simply an \emph{$n$-tree}) is a tree with
	marked vertices $v_1,\ldots,v_n$, along with some number of unmarked
	and unlabeled vertices, such that all vertices of valence 
	$1$ or $2$ are marked. While there is no stated restriction on
	the unmarked vertices, one can show that 
	an $n$-tree has at most $2n-2$ vertices in total. A
	\emph{planar $n$-tree} is the free isotopy class of a 
	planar embedding $\Gamma \to \mathbb{C}$, where $\Gamma$ is
	an $n$-tree. In other words, a planar $n$-tree consists of an $n$-tree
	together with the cyclic counter-clockwise ordering of edges 
	incident to a vertex.
	
	There is a natural partial order on the set of planar $n$-trees,
	defined by declaring that $\Gamma_1 \leq \Gamma_2$ if there is a 
	collection of subtrees in $\Gamma_1$, each of which contains at most 
	one marked vertex, such that contracting each subtree to a point yields 
	$\Gamma_2$. Note that the direction of this partial order is perhaps the reverse
	of what you might expect; unlike the poset structure for Outer space
	\cite{culler86}, for example, contracting a subtree corresponds to 
	moving up in the partial order. Within this partially ordered set, each 
	planar $n$-tree has a corresponding \emph{lower set} which consists of
	all elements below that tree in the partial order; \emph{upper sets} 
	are defined similarly. Our first main theorem identifies the lower sets as 
	familiar combinatorial objects.
	
	\begin{thma}[Theorem~\ref{thm:poset-product}]\label{thma:face-poset-polytope}
		The lower set of a planar $n$-tree is the face poset of a convex
		polytope which can be expressed as a product of associahedra and cyclohedra.
	\end{thma}

	Given a fixed set of points $z_1,\ldots,z_n$ in the complex plane, analogously
	define a \emph{planted $n$-tree} to be the \emph{relative} isotopy class
	of a planar embedding $\Gamma\to\mathbb{C}$ where the marked vertex
	$v_i$ is sent to $z_i$ and the images of the marked points are fixed under 
	isotopy. One can similarly define a partial order by contraction on the set of
	planted $n$ trees and refer to it as the \emph{planted tree poset}.
	
	The inspiration for studying $n$-trees comes in part from 
	complex dynamics. The \emph{Hubbard tree} for a postcritically finite
	polynomial \cite{douady-hubbard-1,douady-hubbard-2} can be viewed
	as representing a planted $n$-tree where the $n$ marked points
	correspond to the postcritical set for the polynomial. 
	In a recent article on complex dynamics by Belk, Lanier, Margalit, and Winarski 
	\cite{belk}, the authors define the \emph{(simplicial) tree complex} as 
	the geometric realization of the planted tree poset and
	use this complex to study Hubbard trees. 
	
	The pure mapping class group for the $n$-punctured plane acts naturally
	on the set of planted $n$-trees (and thus
	the simplicial tree complex), and there is a one-to-one correspondence
	between the orbits of this action and the set of planar $n$-trees. Moreover,
	this action is equivariant with respect to the partial order,
	so each lower set in the planted tree poset may be
	interpreted using Theorem~\ref{thma:face-poset-polytope}. As a 
	consequence, the simplicial tree complex can be viewed as the
	result of subdividing a simpler polyhedral cell complex.
	
	\begin{thma}[Theorem~\ref{thm:polyhedral-tree-complex}]\label{thma:polyhedral-tree-complex}
		The planted tree poset is the face poset of a
		cell complex which we call the polyhedral tree complex. Each
		cell can be expressed as a product of associahedra and cyclohedra;
		in particular, the top dimensional cells are all products of cyclohedra. 
		Furthermore, the simplicial tree complex is the
		barycentric subdivision of the polyhedral tree complex.
	\end{thma}

	Finally, define a planted $n$-tree to be \emph{reduced} 
	if it has no edges between unmarked vertices. This leads to an equivalence 
	relation on the set of planted  $n$-trees by declaring that two trees are 
	equivalent if contracting all the edges between unmarked vertices in 
	each one yields the same reduced tree.
	Each equivalence class for this relation is then canonically labeled by
	a reduced tree, and the original partial order on planted $n$-trees 
	induces a partial order on the set of equivalence classes. 
	The final theorem examines the structure of this poset and demonstrates
	a connection with the noncrossing hypertree poset introduced in \cite{hypertrees}.

	\begin{thma}[Theorems~\ref{thm:lower-reduced} and \ref{thm:upper-reduced}]
		\label{thma:reduced}
		Let $\Gamma$ be a reduced $n$-tree with equivalence class $[\Gamma]$.
		Then the lower set of $[\Gamma]$ 
		is isomorphic to the face poset for a product of simplices and the upper set of 
		$[\Gamma]$ is isomorphic to a product of noncrossing hypertree posets. 
	\end{thma}
	
	The combinatorial transition from planted trees to reduced trees
	has a topological interpretation in which the polyhedral tree complex is
	analogously transformed into a polysimplicial complex. This resulting 
	complex is closely connected to both the cactus complex \cite{nekrashevych14} 
	and the dual braid complex \cite{brady01,brady10};
	the details will be given in a future article.
	
	This article begins with some preliminaries on posets in Section~1, 
	followed in Section~2 by the introduction of planar $n$-trees and their
	contractions. Section~3 defines associahedra and cyclohedra and
	uses them to prove Theorem~\ref{thma:face-poset-polytope}.
	Section~4 concerns planted $n$-trees and the proof of 
	Theorem~\ref{thma:polyhedral-tree-complex}. Finally, reduced
	$n$-trees and are discussed and Theorem~\ref{thma:reduced} is
	proven in Section~5. \\
	
	\section{Preliminaries}
	
	We begin with a few useful facts and examples regarding
	partially ordered sets and their associated cell complexes.
	Given a cell complex $X$, the \emph{face poset} $P(X)$ is
	the set of all nonempty faces of $X$, partially ordered
	by inclusion. Given a poset $P$, the \emph{order complex} 
	(or \emph{geometric realization}) $\Delta(P)$
	is the simplicial complex with vertices corresponding to
	elements of $P$ and a $k$-simplex on vertices $x_1,\ldots, x_{k+1}$
	for each chain $x_1 \leq \cdots \leq x_{k+1}$ in $P$. 
	These two operations are not quite inverses of one another, but
	they are closely related: if $X$ is a
	polytopal cell complex (its cells are convex polytopes which intersect in
	smaller-dimensional convex polytopes), 
	then the simplicial complex $\Delta(P(X))$ is the
	barycentric subdivision of $X$.
	For an in-depth exploration of these tools, see \cite{wachs07}.
	
	The	\emph{Boolean lattice}, denoted $\bool_n$, is the set of all subsets of
	$\{1,2,\ldots,n\}$, partially ordered by inclusion.
	As a useful shorthand, let $\bool_n^\ast$ denote
	the subposet of all nonempty subsets.
	
	\begin{example}\label{ex:boolean-simplex}
		The $n$-dimensional simplex may be realized as the set of 
		all points $(x_1,\ldots,x_{n+1}) \in \mathbb{R}^{n+1}$ with 
		positive entries such that $x_1 + \cdots + x_{n+1} = 0$. Then
		each face of dimension $k$ may be described by specifying
		a nonempty subset of $k$ coordinates which are nonzero, and it 
		follows that the face poset of the $n$-simplex is isomorphic 
		to $\bool_n^\ast$. Furthermore, the order complex of 
		$\bool_n^\ast$ is a barycentrically subdivided $n$-simplex. 
		Each subset of size $k$ labels the barycenter of a 
		$k$-dimensional face, and each of the $n!$ maximal chains in 
		$\bool_n^\ast$ labels an $n$-simplex in the new simplicial 
		cell structure. See Figure~\ref{fig:boolean_simplex} for an 
		example in dimension $2$.
	\end{example}
	
	\begin{figure}
		\centering
		\begin{tikzpicture}

	\begin{scope}[shift={(-7,0)}]
	\foreach \i in {1,...,3} {
		\coordinate (v\i) at (\i*120+90:2cm);
	}
	\filldraw[GreenPoly] (v1)--(v2)--(v3)--cycle;

	\node () at (0,0) {$\{1,2,3\}$};
	
	\node () at (155:1.8cm) {$\{1,2\}$};
	\node () at (270:1.4cm) {$\{2,3\}$};
	\node () at (25:1.8cm) {$\{1,3\}$};
	
	\node () at (90:2.3cm) {$\{1\}$};
	\node () at (210:2.5cm) {$\{2\}$};
	\node () at (330:2.5cm) {$\{3\}$};
	\end{scope}
	
	
	\draw[BluePoly,->] (-4,0.5) -- (-3,0.5);
	
	\node (top) at (0,2) {$\{1,2,3\}$};
	
	\node (a1) at (-2,0.5) {$\{1,2\}$};
	\node (a2) at (0,0.5) {$\{1,3\}$};
	\node (a3) at (2,0.5) {$\{2,3\}$};
	
	\node (b1) at (-2,-1) {$\{1\}$};
	\node (b2) at (0,-1) {$\{2\}$};
	\node (b3) at (2,-1) {$\{3\}$};
	
	\draw (a1) -- (b1) -- (a2) -- (b3) -- (a3) -- (b2) -- (a1);
	
	\draw (top) -- (a1);
	\draw (top) -- (a2);
	\draw (top) -- (a3);
	
	\draw[BluePoly,->] (3,0.5) -- (4,0.5);
	
	
	\begin{scope}[shift={(7,0)}]
	\foreach \i in {1,...,3} {
		\coordinate (v\i) at (\i*120+90:2cm);
	}
	\filldraw[GreenPoly] (v1)--(v2)--(v3)--cycle;
	\draw[GreenPoly] (v1) -- (30:1cm);
	\draw[GreenPoly] (v2) -- (150:1cm);
	\draw[GreenPoly] (v3) -- (270:1cm);
	
	\node () at (-5:2.3cm) {$\{1,2,3\}$};
	\draw[GrayPoly,->] (-7:1.4cm) -- (-5:0.2cm);
	
	\node () at (155:1.8cm) {$\{1,2\}$};
	\node () at (270:1.4cm) {$\{2,3\}$};
	\node () at (25:1.8cm) {$\{1,3\}$};
	
	\node () at (90:2.3cm) {$\{1\}$};
	\node () at (210:2.5cm) {$\{2\}$};
	\node () at (330:2.5cm) {$\{3\}$};
	\end{scope}
\end{tikzpicture}
		\caption{The face poset of a $2$-simplex is isomorphic to 
			$\bool_3^\ast$, and its order complex is a barycentrically
			subdivided $2$-simplex.\label{fig:boolean_simplex}}
	\end{figure}
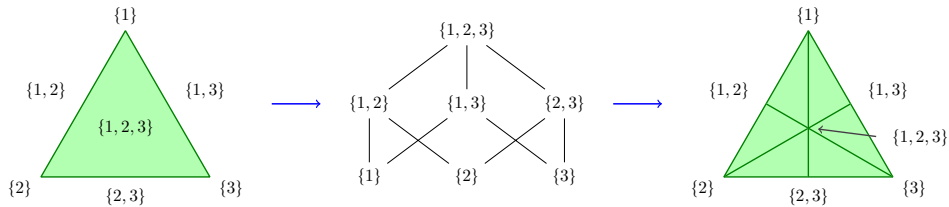
	
	\begin{defn}\label{def:upper-lower}
		If $P$ is a poset with $a,b\in P$ and $a\leq b$,
		then the \emph{interval} between $a$ and $b$, denoted $[a,b]$, 
		is the subposet of all $x\in P$ such that $a\leq x\leq b$.
		The \emph{upper set} of $a$, denoted $\upper(a)$, is the
		set of all $x \in P$ such that $a \leq x$, and the
		\emph{lower set} of $b$, denoted $\low(b)$, is the
		set of all $x \in P$ such that $x \leq b$.
	\end{defn}
	
	\begin{example}
		Let $A$ and $B$ be elements of $\bool_n$ such that $A \leq B$, 
		$|A| = k$, and $|B| = \ell$. 
		Then the interval $[A,B]$ consists of all subsets of $B$ 
		which contain $A$, and so $[A,B]$ is isomorphic to $\bool_{\ell-k}$. 
		Similarly, the lower set $\low(A)$ is isomorphic to $\bool_k$ and 
		the upper set $\upper(A)$ is isomorphic to $\bool_{n-k}$.
	\end{example}
	
	\section{Planar Trees}
	
	This section introduces planar $n$-trees and an associated partial order.
	
	\begin{defn}
		A \emph{tree} is a finite, contractible, 
		$1$-dimensional simplicial complex. Connected subcomplexes
		of a tree are \emph{subtrees} and a collection of 
		disjoint subtrees is a \emph{subforest}.
		The number of edges incident to a vertex $v$ is
		referred to as its \emph{valence} $\val(v)$; a vertex is a
		\emph{leaf} if it has valence $1$ and an \emph{interior vertex}
		otherwise. Given any two vertices $v_1$ and $v_2$ in a tree,
		there is a unique path subcomplex from $v_1$ to $v_2$, which
		we refer to as a \emph{geodesic} and denote by $\gamma(v_1,v_2)$.
	\end{defn}
		
	\begin{defn}
		Given an integer $n\geq 3$, a \emph{marked $n$-tree} 
		(or simply an \emph{$n$-tree}) consists of a tree 
		$\Gamma$ with a collection of marked vertices 
		$v_1,\ldots,v_n$ which includes each vertex of valence
		$1$ or $2$ in $\Gamma$; see Figure~\ref{fig:n-tree} for an example
		with $n=6$.		
		Since every leaf of an $n$-tree is labeled,
		the identity map is the only 
		isomorphism between $n$-trees which preserves the labels on marked
		vertices. Using the Euler characteristic, one can show
		that an $n$-tree has between $n$ and $2n-2$ vertices.
	\end{defn}

	\begin{figure}
		\centering
			
	\tikzstyle{dot}=[shape=circle,draw,color=black,fill=black,inner sep=3pt]
	
	\begin{tikzpicture}
	\begin{scope}[shift={(0,0)}]
	\filldraw[thick, fill = yellow!40] (0,0) circle (2.2cm);
	\foreach \n in {1,2,...,5} {
		\node [dot,color=blue!70] (v\n) at (\n*72:1.5cm) {};	
		\node () at (\n*72:1.9cm) {$\n$};	
	}
	\node [dot,color=blue!70] (v6) at (-0.6,0) {};
	\node () at (-0.45,0.3) {$6$};
	\node [dot,fill=white] (u) at (0.3,0) {};
	\begin{scope}[ultra thick]
	\draw (v2) -- (v6) -- (v3);
	\draw (v6) -- (u) -- (v5);
	\draw (v1) -- (u) -- (v4);
	\end{scope}
	\end{scope}
	\end{tikzpicture}
	
		\caption{A marked $6$-tree}
		\label{fig:n-tree}
	\end{figure}
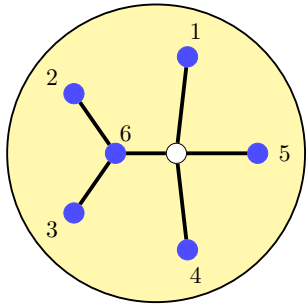
	
	\begin{defn}
		Let $\Gamma_1$ and $\Gamma_2$ be $n$-trees with marked vertices
		$v_1,\ldots,v_n$. A \emph{contraction}
		$f\colon\Gamma_1\to\Gamma_2$ is a surjective cellular map with
		the property that, for all $i$, $f^{-1}(v_i)$ is a subtree of $\Gamma_1$
		which contains $v_i$.
		Regarding contractions up to isotopy within each edge, each
		such map can be determined in a purely combinatorial manner.
		In other words, a contraction is obtained from an $n$-tree
		by specifying some number of subtrees, each of which contains
		at most one marked vertex, and retracting each subtree to a point.
	\end{defn}
	
	In more general cases (e.g. if unmarked vertices of
	valence $2$ were allowed), there may be several different contractions from one 
	given tree to another. With $n$-trees, however, the existence of 
	contractions is far more restrictive.
	
	\begin{lem}\label{lem:unique-contraction}
		Let $\Gamma_1$ and $\Gamma_2$ be $n$-trees. If there is a
		contraction $\Gamma_1 \to \Gamma_2$, then it is unique.
		Consequently, contracting two different subforests of the
		same $n$-tree must result in different trees.
	\end{lem}
	
	\begin{proof}
		Suppose that $f\colon \Gamma_1 \to \Gamma_2$ is a contraction. 
		By definition, $f$ sends each marked vertex in $\Gamma_1$
		to the corresponding marked vertex in $\Gamma_2$. We will show that
		the image of each unmarked vertex is also completely determined.
		
		If $u$ is an unmarked vertex, then we know that the 
		valence of $u$ is at least three, and so the complement of $u$ in 
		$\Gamma_1$ consists of at least three connected 
		components. Fix $v_i$, $v_j$, and $v_k$ to be marked vertices in three
		distinct connected components of $\Gamma_1 - \{u\}$. By construction,
		we know that $u$ is the unique vertex which lies in the common 
		intersection of the geodesics $\gamma(v_i,v_j)$, $\gamma(v_j,v_k)$, and $\gamma(v_i,v_k)$;
		if there were another, then $\Gamma_1$ would have a cycle.
		
		Now, since $f$ is a contraction, we know that $f$ sends $v_i$, $v_j$, and $v_k$
		to the corresponding vertices in $\Gamma_2$ and that the images 
		$f(\gamma(v_i,v_j))$, $f(\gamma(v_j,v_k))$, and 
		$f(\gamma(v_i,v_k))$ yield three paths
		between these three vertices. Since $\Gamma_2$ is also a tree, the 
		common intersection of these paths contains a single point,
		which must be $f(u)$. Thus, the image of each vertex under $f$ is completely determined by
		$\Gamma_1$ and $\Gamma_2$, so the contraction is unique.
	\end{proof}

	As a consequence of Lemma~\ref{lem:unique-contraction},
	contractions of a given $n$-tree correspond exactly to subforests
	for which each subtree contains at most one marked point.

	\begin{defn}\label{def:singly-marked}
		A subtree of an $n$-tree $\Gamma$ is \emph{unmarked} if it contains
		no marked vertices and \emph{singly-marked} if it contains exactly one.
		That is, a subforest of $\Gamma$ corresponds to a contraction of $\Gamma$
		if and only if its components are all unmarked or singly-marked subtrees.
		If $\Gamma_1 \to \Gamma_2$ is a contraction of $n$-trees, 
		define $F(\Gamma_1,\Gamma_2)$ to be the unique subforest of $\Gamma_1$
		consisting of unmarked and/or singly-marked subtrees 
		which can be contracted to obtain $\Gamma_2$.
	\end{defn}

	Before moving on, we record a useful corollary of 
	Lemma~\ref{lem:unique-contraction} and Definition~\ref{def:singly-marked}.

	\begin{cor}\label{cor:subforests}
		Let $\Gamma$, $\Gamma_1$, and $\Gamma_2$ be $n$-trees.
		\begin{enumerate}
			\item If there are contractions 
			$\Gamma_1 \to \Gamma$ and $\Gamma_2 \to \Gamma$, then
			there is a contraction $\Gamma_1\to \Gamma_2$ if and only if
			$F(\Gamma_2,\Gamma)$ is a subforest of $F(\Gamma_1,\Gamma)$.
			\item If there are contractions 
			$\Gamma \to \Gamma_1$ and $\Gamma \to \Gamma_2$, then
			there is a contraction $\Gamma_1\to \Gamma_2$ if and only if
			$F(\Gamma,\Gamma_1)$ is a subforest of $F(\Gamma,\Gamma_2)$.
		\end{enumerate}
	\end{cor}
	
	We are interested in two types of planar embeddings for 
	marked $n$-trees. First, we consider the case of planar
	embeddings up to free isotopy. In Section~\ref{sec:tree-complexes},
	we instead examine embeddings up to relative isotopy fixing the 
	marked points.
	
	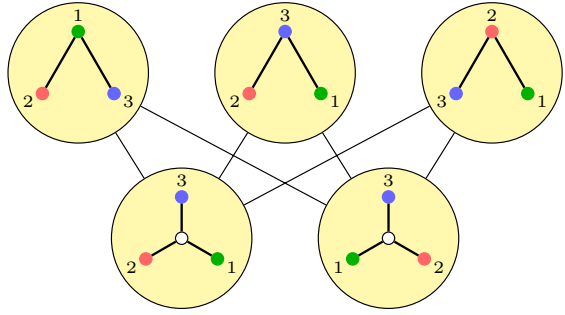
\begin{figure}
		\centering
		\begin{tikzpicture}

	\begin{scope}[shift={(-2.5,2)}]
		\node [shape=circle,draw,fill=yellow!40,inner sep = 0.6cm] (pb1) at (0,0) {};
		\node [dot,color=black!30!green] (1) at (90:0.5cm) {};
		\node [dot,color=red!60] (2) at (210:0.5cm) {};
		\node [dot,color=blue!60] (3) at (330:0.5cm) {};
		\node () at (90:0.7cm) {\tiny $1$};	
		\node () at (210:0.7cm) {\tiny $2$};	
		\node () at (330:0.7cm) {\tiny $3$};		
		\begin{scope}[thick]
			\draw (3) -- (1) -- (2);
		\end{scope}
	\end{scope}
	
	\begin{scope}[shift={(0,2)}]
		\node [shape=circle,draw,fill=yellow!40,inner sep = 0.6cm] (pb2) at (0,0) {};
		\node [dot,color=blue!60] (1) at (90:0.5cm) {};
		\node [dot,color=red!60] (2) at (210:0.5cm) {};
		\node [dot,color=black!30!green] (3) at (330:0.5cm) {};
		\node () at (90:0.7cm) {\tiny $3$};	
		\node () at (210:0.7cm) {\tiny $2$};	
		\node () at (330:0.7cm) {\tiny $1$};
		\begin{scope}[thick]
			\draw (3) -- (1) -- (2);
		\end{scope}
	\end{scope}
	
	\begin{scope}[shift={(2.5,2)}]
		\node [shape=circle,draw,fill=yellow!40,inner sep = 0.6cm] (pb3) at (0,0) {};
		\node [dot,color=red!60] (1) at (90:0.5cm) {};
		\node [dot,color=blue!60] (2) at (210:0.5cm) {};
		\node [dot,color=black!30!green] (3) at (330:0.5cm) {};
		\node () at (90:0.7cm) {\tiny $2$};	
		\node () at (210:0.7cm) {\tiny $3$};	
		\node () at (330:0.7cm) {\tiny $1$};
		\begin{scope}[thick]
			\draw (3) -- (1) -- (2);
		\end{scope}
	\end{scope}
	
	\begin{scope}[shift={(-1.25,0)}]
		\node [shape=circle,draw,fill=yellow!40,inner sep = 0.6cm] (pa1) at (0,0) {};
		\node [dot,color=blue!60] (1) at (90:0.5cm) {};
		\node [dot,color=red!60] (2) at (210:0.5cm) {};
		\node [dot,color=black!30!green] (3) at (330:0.5cm) {};
		\node [dot,fill=white] (4) at (0,0) {};
		\node () at (90:0.7cm) {\tiny $3$};	
		\node () at (210:0.7cm) {\tiny $2$};	
		\node () at (330:0.7cm) {\tiny $1$};
		\begin{scope}[thick]
			\draw (3) -- (4);
			\draw (2) -- (4);
			\draw (1) -- (4);
		\end{scope}
	\end{scope}
	
	\begin{scope}[shift={(1.25,0)}]
		\node [shape=circle,draw,fill=yellow!40,inner sep = 0.6cm] (pa2) at (0,0) {};
		\node [dot,color=blue!60] (1) at (90:0.5cm) {};
		\node [dot,color=black!30!green] (2) at (210:0.5cm) {};
		\node [dot,color=red!60] (3) at (330:0.5cm) {};
		\node [dot,fill=white] (4) at (0,0) {};
		\node () at (90:0.7cm) {\tiny $3$};	
		\node () at (210:0.7cm) {\tiny $1$};	
		\node () at (330:0.7cm) {\tiny $2$};
		\begin{scope}[thick]
			\draw (3) -- (4);
			\draw (2) -- (4);
			\draw (1) -- (4);
		\end{scope}
	\end{scope}

	\draw[line join = bevel] (pa1) -- (pb1) -- (pa2) -- (pb2) -- (pa1) -- (pb3) -- (pa2);

\end{tikzpicture}
		\caption{The planar tree poset $\planar_3$}
		\label{fig:p3_poset}
	\end{figure}
	
	\begin{defn}
		Let $\Gamma$ be an $n$-tree. A \emph{planar $n$-tree}
		is a free isotopy class of planar embeddings 
		$\phi\colon \Gamma \to \mathbb{C}$. That is,
		a planar $n$-tree is just an $n$-tree together with
		a cyclic counter-clockwise ordering of the edges 
		incident to each vertex. When the context is clear, 
		we omit explicit references to embeddings or markings, and instead refer to 
		a planar $n$-tree simply as $\Gamma$.
	\end{defn}

	Planar $n$-trees are distinct from those studied in \cite{billera01},
	for example, since we include the ordering of edges at each
	vertex.
	
	\begin{defn}\label{def:contraction-planar}
		Let $\phi_1\colon\Gamma_1\to \mathbb{C}$ and
		$\phi_2\colon\Gamma_2\to\mathbb{C}$ be planar $n$-trees.
		If $f\colon\Gamma_1\to\Gamma_2$ is a contraction
		such that $\phi_1$ and $\phi_2\circ f$ are isotopic, 
		then we say that $f$ is a \emph{planar contraction}.
		If this is the case, we abuse notation to omit the embedding 
		and write $\Gamma_1\leq\Gamma_2$. This determines
		a partial order on the set of all $n$-trees which we 
		call the \emph{planar tree poset} and
		denote by $\planar_n$. See Figure~\ref{fig:p3_poset} for
		an example when $n=3$.
	\end{defn}
	
	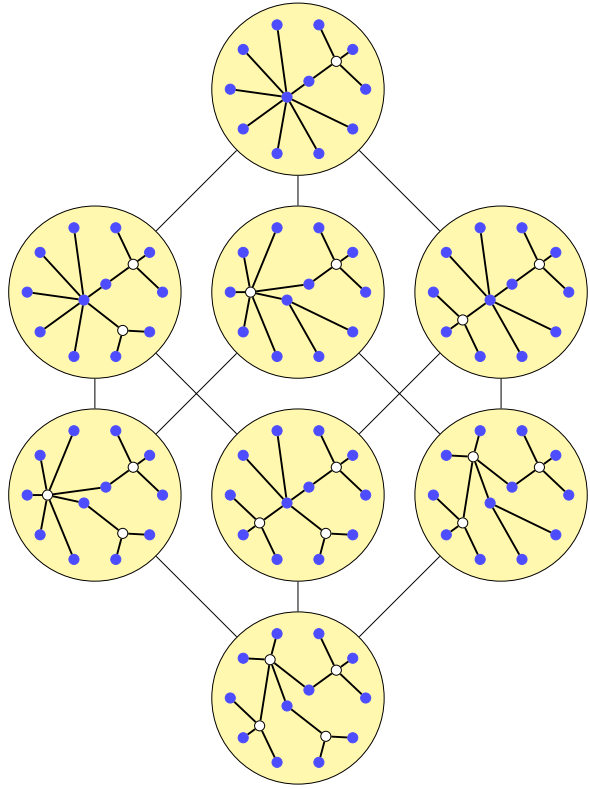
\begin{figure}
		\centering
		\begin{tikzpicture}

	
	\node [shape=circle,draw,fill=yellow!40,inner sep = 0.9cm] (pa) at (0,0) {};
	
	\foreach \i in {1,2,3} {
		\node [shape=circle,draw,fill=yellow!40,inner sep = 0.9cm] (pb\i) at (-6+3*\i,3) {};
	}

	\foreach \i in {1,2,3} {
		\node [shape=circle,draw,fill=yellow!40,inner sep = 0.9cm] (pc\i) at (-6+3*\i,6) {};
	}

	\node [shape=circle,draw,fill=yellow!40,inner sep = 0.9cm] (pd) at (0,9) {};
	
	
	\begin{scope}[shift={(0,9)}]
		\foreach \i in {1,...,10} {
			\node [dot,color=blue!70] (v\i)	at (36*\i:1cm) {};
		}
		\node [dot,fill=white] (u1) at (36:0.7cm) {};
		\node [dot,color=blue!70] (c1) at (36:-0.2cm) {};
		\node [dot,color=blue!70] (c2) at (36:0.2cm) {};		
		
		\begin{scope}[thick]
			\draw (c1) -- (v3);
			\draw (c1) -- (v4);
			\draw (c1) -- (v5);
			\draw (c1) -- (v6);
			\draw (c1) -- (v7);
			\draw (c1) -- (v8);
			\draw (c1) -- (v9);
			
			\draw (c1) -- (c2);
			
			\draw (c2) -- (u1);
			
			\draw (u1) -- (v10);
			\draw (u1) -- (v1);
			\draw (u1) -- (v2);
		\end{scope}
	\end{scope}
	
	
	\begin{scope}[shift={(-3,6)}]
	\foreach \i in {1,...,10} {
		\node [dot,color=blue!70] (v\i)	at (36*\i:1cm) {};
	}
	\node [dot,fill=white] (u1) at (36:0.7cm) {};
	\node [dot,color=blue!70] (c1) at (36:-0.2cm) {};
	\node [dot,color=blue!70] (c2) at (36:0.2cm) {};	
	
	\node [dot,fill=white] (u2) at (-54:0.7cm) {};
	
	\begin{scope}[thick]
	\draw (c1) -- (u2);
	
	\draw (c2) -- (u1);
	\draw (c2) -- (c1);
	
	\draw (u1) -- (v10);
	\draw (u1) -- (v1);
	\draw (u1) -- (v2);
	
	\draw (u2) -- (v9);
	\draw (u2) -- (v8);
	
	\draw (c1) -- (v3);
	\draw (c1) -- (v4);
	\draw (c1) -- (v5);
	\draw (c1) -- (v6);
	\draw (c1) -- (v7);
	\end{scope}
	\end{scope}
	
	\begin{scope}[shift={(0,6)}]
	\foreach \i in {1,...,10} {
		\node [dot,color=blue!70] (v\i)	at (36*\i:1cm) {};
	}
	\node [dot,fill=white] (u1) at (36:0.7cm) {};
	\node [dot,color=blue!70] (c1) at (36:-0.2cm) {};
	\node [dot,color=blue!70] (c2) at (36:0.2cm) {};	
	
	\node [dot,fill=white] (u5) at (180:0.7cm) {};
	
	\begin{scope}[thick]
	\draw (c1) -- (u5);
	
	\draw (c2) -- (u1);
	\draw (c2) -- (u5);
	
	\draw (u1) -- (v10);
	\draw (u1) -- (v1);
	\draw (u1) -- (v2);
	
	\draw (c1) -- (v9);
	\draw (c1) -- (v8);
	
	\draw (u5) -- (v3);
	\draw (u5) -- (v4);
	
	\draw (u5) -- (v5);
	\draw (u5) -- (v6);
	\draw (u5) -- (v7);
	\end{scope}
	\end{scope}
	
	\begin{scope}[shift={(3,6)}]
	\foreach \i in {1,...,10} {
		\node [dot,color=blue!70] (v\i)	at (36*\i:1cm) {};
	}
	\node [dot,fill=white] (u1) at (36:0.7cm) {};
	\node [dot,color=blue!70] (c1) at (36:-0.2cm) {};
	\node [dot,color=blue!70] (c2) at (36:0.2cm) {};	
	
	\node [dot,fill=white] (u4) at (216:0.7cm) {};
	
	\begin{scope}[thick]
	\draw (c1) -- (u4);
	
	\draw (c2) -- (u1);
	\draw (c2) -- (c1);
	
	\draw (u1) -- (v10);
	\draw (u1) -- (v1);
	\draw (u1) -- (v2);
	
	\draw (c1) -- (v9);
	\draw (c1) -- (v8);
	
	\draw (c1) -- (v3);
	\draw (c1) -- (v4);
	
	\draw (u4) -- (v5);
	\draw (u4) -- (v6);
	\draw (u4) -- (v7);
	\end{scope}
	\end{scope}
	
	
	\begin{scope}[shift={(-3,3)}]
	\foreach \i in {1,...,10} {
		\node [dot,color=blue!70] (v\i)	at (36*\i:1cm) {};
	}
	\node [dot,fill=white] (u1) at (36:0.7cm) {};
	\node [dot,color=blue!70] (c1) at (36:-0.2cm) {};
	\node [dot,color=blue!70] (c2) at (36:0.2cm) {};	
	
	\node [dot,fill=white] (u2) at (-54:0.7cm) {};
	\node [dot,fill=white] (u5) at (180:0.7cm) {};
	
	\begin{scope}[thick]
	\draw (c1) -- (u2);
	\draw (c1) -- (u5);
	
	\draw (c2) -- (u1);
	\draw (c2) -- (u5);
	
	\draw (u1) -- (v10);
	\draw (u1) -- (v1);
	\draw (u1) -- (v2);
	
	\draw (u2) -- (v9);
	\draw (u2) -- (v8);
	
	\draw (u5) -- (v3);
	\draw (u5) -- (v4);
	
	\draw (u5) -- (v5);
	\draw (u5) -- (v6);
	\draw (u5) -- (v7);
	\end{scope}
	\end{scope}
	
	\begin{scope}[shift={(0,3)}]
	\foreach \i in {1,...,10} {
		\node [dot,color=blue!70] (v\i)	at (36*\i:1cm) {};
	}
	\node [dot,fill=white] (u1) at (36:0.7cm) {};
	\node [dot,color=blue!70] (c1) at (36:-0.2cm) {};
	\node [dot,color=blue!70] (c2) at (36:0.2cm) {};	
	
	\node [dot,fill=white] (u2) at (-54:0.7cm) {};
	\node [dot,fill=white] (u4) at (216:0.7cm) {};
	
	\begin{scope}[thick]
	\draw (c1) -- (u2);
	\draw (c1) -- (u4);
	
	\draw (c2) -- (u1);
	\draw (c2) -- (c1);
	
	\draw (u1) -- (v10);
	\draw (u1) -- (v1);
	\draw (u1) -- (v2);
	
	\draw (u2) -- (v9);
	\draw (u2) -- (v8);
	
	\draw (c1) -- (v3);
	\draw (c1) -- (v4);
	
	\draw (u4) -- (v5);
	\draw (u4) -- (v6);
	\draw (u4) -- (v7);
	\end{scope}
	\end{scope}
	
	\begin{scope}[shift={(3,3)}]
	\foreach \i in {1,...,10} {
		\node [dot,color=blue!70] (v\i)	at (36*\i:1cm) {};
	}
	\node [dot,fill=white] (u1) at (36:0.7cm) {};
	\node [dot,color=blue!70] (c1) at (36:-0.2cm) {};
	\node [dot,color=blue!70] (c2) at (36:0.2cm) {};	
	
	\node [dot,fill=white] (u3) at (126:0.7cm) {};
	\node [dot,fill=white] (u4) at (216:0.7cm) {};
	
	\begin{scope}[thick]
	\draw (c1) -- (u3);
	\draw (u3) -- (u4);
	
	\draw (c2) -- (u1);
	\draw (c2) -- (u3);
	
	\draw (u1) -- (v10);
	\draw (u1) -- (v1);
	\draw (u1) -- (v2);
	
	\draw (c1) -- (v9);
	\draw (c1) -- (v8);
	
	\draw (u3) -- (v3);
	\draw (u3) -- (v4);
	
	\draw (u4) -- (v5);
	\draw (u4) -- (v6);
	\draw (u4) -- (v7);
	\end{scope}
	\end{scope}
	
	
	\begin{scope}[shift={(0,0)}]
		\foreach \i in {1,...,10} {
			\node [dot,color=blue!70] (v\i)	at (36*\i:1cm) {};
		}
		\node [dot,fill=white] (u1) at (36:0.7cm) {};
		\node [dot,color=blue!70] (c1) at (36:-0.2cm) {};
		\node [dot,color=blue!70] (c2) at (36:0.2cm) {};	
		
		\node [dot,fill=white] (u2) at (-54:0.7cm) {};
		\node [dot,fill=white] (u3) at (126:0.7cm) {};
		\node [dot,fill=white] (u4) at (216:0.7cm) {};
		
		\begin{scope}[thick]
			\draw (c1) -- (u2);
			\draw (c1) -- (u3);
			\draw (u3) -- (u4);
			
			\draw (c2) -- (u1);
			\draw (c2) -- (u3);
			
			\draw (u1) -- (v10);
			\draw (u1) -- (v1);
			\draw (u1) -- (v2);
			
			\draw (u2) -- (v9);
			\draw (u2) -- (v8);
			
			\draw (u3) -- (v3);
			\draw (u3) -- (v4);
			
			\draw (u4) -- (v5);
			\draw (u4) -- (v6);
			\draw (u4) -- (v7);
		\end{scope}
	\end{scope}
	
	
	\begin{scope}[line join = bevel]
		\draw (pa) -- (pb1);
		\draw (pa) -- (pb2);
		\draw (pa) -- (pb3);
		
		\draw (pb1) -- (pc1) -- (pb2) -- (pc3) -- (pb3) -- (pc2) -- (pb1);
		
		\draw (pd) -- (pc1);
		\draw (pd) -- (pc2);
		\draw (pd) -- (pc3);
	\end{scope}

\end{tikzpicture}
		\caption{An interval in $\planar_{12}$. In this figure 
			and several others which follow, we omit the vertex 
			labels and instead distinguish
			marked vertices by their location in the plane.}
		\label{fig:boolean-poset}
	\end{figure}
	
	\begin{thm}
		If $\Gamma_1 \leq \Gamma_2$ in $\planar_n$ and the
		subforest $F(\Gamma_1,\Gamma_2)$ has $k$ edges, then
		the interval $[\Gamma_1,\Gamma_2]$ is
		isomorphic to the Boolean lattice $\bool_k$.
	\end{thm}
	
	\begin{proof}
		Let $\Gamma_1,\Gamma_2\in \planar_n$ with $\Gamma_1\leq \Gamma_2$.
		By the second part of Corollary~\ref{cor:subforests},
		the planar $n$-trees lying between $\Gamma_1$ and $\Gamma_2$ in
		$\planar_n$ correspond exactly to the subforests of $F(\Gamma_1,\Gamma_2)$,
		each of which is chosen by selecting a subset of the $k$ edges.
		This gives a bijection from the interval $[\Gamma_1,\Gamma_2]$ 
		to the Boolean lattice $\bool_k$, and this is an isomorphism 
		since $\Gamma^\prime \leq \Gamma^{\pprime}$ in $[\Gamma_1,\Gamma_2]$
		if and only if $F(\Gamma_1,\Gamma^\prime)$ is a subforest of 
		$F(\Gamma_1,\Gamma^{\pprime})$ by Corollary~\ref{cor:subforests}.
		See Figure~\ref{fig:boolean-poset} for an example.
	\end{proof}
	
	\section{Associahedra and Cyclohedra}
	\label{sec:assoc-cyclo}
	
	The lower sets in the planar tree poset
	are closely related to two important types of
	convex polytopes: associahedra and cyclohedra.
	For an excellent overview of both, 
	see \cite[Section 1]{devadoss03}.
	
	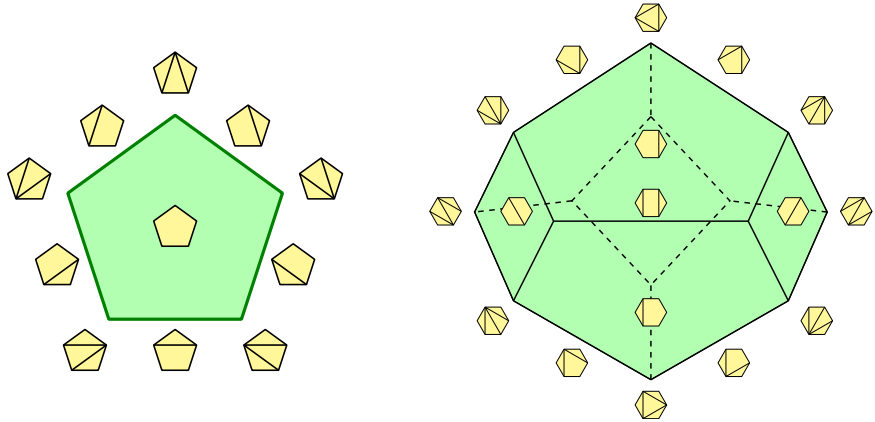
\begin{figure}
		\centering
		\begin{subfigure}{.5\textwidth}
			\centering
			\begin{tikzpicture}
	\newcommand{\hexback}{
		\foreach \n in {1,...,6} {
			\coordinate (\n) at  (\n*72+18:0.2cm);
		}
		\draw[fill=yellow!50] (1)--(2)--(3)--(4)--(5)--cycle;
	}

	\foreach \i in {1,...,5} {
		\coordinate (v\i) at (\i*72+18:1cm);
	}
	\filldraw[GreenPoly] (v1)--(v2)--(v3)--(v4)--(v5)--cycle;

	\begin{scope}
		\coordinate (top) at (0,0);
		\foreach \i in {1,...,5} \coordinate (a\i) at (\i*72+54:1.1cm);
		\foreach \i in {1,...,5} \coordinate (b\i) at (\i*72+18:1.3597cm);
	\end{scope}
	
	\begin{scope}[line join=bevel]
	
		\begin{scope}[shift={(top)}]
			\hexback
		\end{scope}
		
		\begin{scope}[shift={(a1)}]
			\hexback
			\draw (1) -- (3);
		\end{scope}
		
		\begin{scope}[shift={(a2)}]
			\hexback
			\draw (3) -- (5);
		\end{scope}
		
		\begin{scope}[shift={(a3)}]
			\hexback
			\draw (5) -- (2);
		\end{scope}
		
		\begin{scope}[shift={(a4)}]
			\hexback
			\draw (2) -- (4);
		\end{scope}
		
		\begin{scope}[shift={(a5)}]
			\hexback
			\draw (4) -- (1);
		\end{scope}
		
		\begin{scope}[shift={(b1)}]
			\hexback
			\draw (4) -- (1) -- (3);
		\end{scope}
		
		\begin{scope}[shift={(b2)}]
			\hexback
			\draw (1) -- (3) -- (5);
		\end{scope}
		
		\begin{scope}[shift={(b3)}]
			\hexback
			\draw (3) -- (5) -- (2);
		\end{scope}
		
		\begin{scope}[shift={(b4)}]
			\hexback
			\draw (5) -- (2) -- (4);
		\end{scope}
		
		\begin{scope}[shift={(b5)}]
			\hexback
			\draw (2) -- (4) -- (1);
		\end{scope}
	
	\end{scope}
	
\end{tikzpicture}
		\end{subfigure}%
		\begin{subfigure}{.5\textwidth}
			\centering
			\begin{tikzpicture}

	\begin{scope}[x={(3cm,0cm)},y={(0cm,0.2cm)},z={(0cm,3.2cm)}]

	\newcommand\x{1}
	\newcommand\y{0.5}
	\newcommand\z{0.87}
	

	\coordinate (a) at (0,0,2);
	
	\coordinate (b1) at (0,\x,1.5);
	\coordinate (b2) at (\z,-\y,1.5);
	\coordinate (b3) at (-\z,-\y,1.5);
	
	\coordinate (c1) at (1.116,-0.067,1);
	\coordinate (c2) at (0.5,1,1);
	\coordinate (c3) at (-0.5,1,1);
	\coordinate (c4) at (-1.116,-0.067,1);
	\coordinate (c5) at (-0.616,-0.933,1);
	\coordinate (c6) at (0.616,-0.933,1);
	
	\coordinate (d1) at (0,\x,0.5);
	\coordinate (d2) at (\z,-\y,0.5);
	\coordinate (d3) at (-\z,-\y,0.5);
	
	\coordinate (e) at (0,0,0);
	
	
	\filldraw[fill=green!30,join=bevel] (a) -- (b2) -- (c1) -- (d2) -- (e) -- (d3) -- (c4) -- (b3) -- cycle;
	
	
	\begin{scope}[thick]
	\draw[dashed] (a) -- (b1);
	\draw (a) -- (b2);
	\draw (a) -- (b3);
	\draw[dashed] (b1) -- (c2) -- (d1);
	\draw[dashed] (b1) -- (c3) -- (d1);
	\draw (b2) -- (c6) -- (d2);
	\draw (b2) -- (c1) -- (d2);
	\draw (b3) -- (c4) -- (d3);
	\draw (b3) -- (c5) -- (d3);
	\draw[dashed] (e) -- (d1);
	\draw (e) -- (d2);
	\draw (e) -- (d3);
	\draw[dashed] (c1) -- (c2);
	\draw[dashed] (c3) -- (c4);
	\draw (c5) -- (c6);
	\end{scope}
	
	\newcommand{\hexback}{
		\foreach \n in {1,...,6} {
			\coordinate (\n) at  (\n*60:0.3cm);
		}
		\draw[fill=yellow!50] (1)--(2)--(3)--(4)--(5)--(6)--cycle;
	}
	
	\begin{scope}[line join = bevel]
	\begin{scope}[shift={(0,0,-0.15)}]
		\hexback
		\draw (2) -- (4) -- (6) -- cycle;
	\end{scope}
	
	\begin{scope}[shift={(-0.5,0,0.1)}]
		\hexback
		\draw (6) -- (2) -- (4);
	\end{scope}
	
	\begin{scope}[shift={(0.525,0,0.1)}]
		\hexback
		\draw (2) -- (4) -- (6);
	\end{scope}
	
	\begin{scope}[shift={(0,0,2.15)}]
		\hexback
		\draw (1) -- (3) -- (5) -- cycle;
	\end{scope}
	
	\begin{scope}[shift={(-0.5,0,1.9)}]
	\hexback
	\draw (3) -- (5) -- (1);
	\end{scope}
	
	\begin{scope}[shift={(0.525,0,1.9)}]
	\hexback
	\draw (5) -- (1) -- (3);
	\end{scope}
	
	\begin{scope}[shift={(-1,0,0.35)}]
		\hexback
		\draw (6) -- (2) -- (5);
		\draw (2) -- (4);
	\end{scope}
	
	\begin{scope}[shift={(0,0,0.4)}]
		\hexback
		\draw (2) -- (4);
	\end{scope}
	
	\begin{scope}[shift={(0,0,1.4)}]
		\hexback
		\draw (1) -- (5);
	\end{scope}
	
	\begin{scope}[shift={(0,0,1.05)}]
		\hexback
		\draw (1) -- (5);
		\draw (2) -- (4);
	\end{scope}
	
	\begin{scope}[shift={(1.05,0,0.35)}]
		\hexback
		\draw (4) -- (1);
		\draw (4) -- (2);
		\draw (4) -- (6);
	\end{scope}
	
	\begin{scope}[shift={(-1.3,0,1)}]
		\hexback
		\draw (6) -- (2) -- (5) -- (3);
	\end{scope}
	
	\begin{scope}[shift={(-0.85,0,1)}]
		\hexback
		\draw (2) -- (5) ;
	\end{scope}
	
	\begin{scope}[shift={(0.9,0,1)}]
		\hexback
		\draw (1) -- (4) ;
	\end{scope}
	
	\begin{scope}[shift={(1.3,0,1)}]
		\hexback
		\draw (3) -- (1) -- (4) -- (6);
	\end{scope}
	
	\begin{scope}[shift={(-1,0,1.6)}]
		\hexback
		\draw (1) -- (5);
		\draw (2) -- (5) -- (3);
	\end{scope}
	
	\begin{scope}[shift={(1.05,0,1.6)}]
		\hexback
		\draw (1) -- (3);
		\draw (1) -- (4);
		\draw (1) -- (5);
	\end{scope}
	\end{scope}

	\end{scope}
	
\end{tikzpicture}
		\end{subfigure}
		\caption{The associahedra $K_4$ and $K_5$, with some of the
			front-facing cells labeled}
		\label{fig:assoc_poly}
	\end{figure}
	
	\begin{defn}\label{def:triangulation}
		Let $n\geq 3$, define $z_k = e^{i\pi k / n}$ for
		every integer $k$, and let $P_n$ denote the regular $n$-gon
		obtained by taking the convex hull in $\mathbb{C}$ of
		the points $z_1,\ldots,z_n$. A straight line segment between
		non-adjacent vertices of $P_n$ is called a
		\emph{diagonal}, and a \emph{partial triangulation} of
		$P_n$ is a (possibly empty) set of pairwise
		non-intersecting diagonals. 
		The set $\tri_n$ of all partial triangulations of $P_n$
		is a partially ordered set under reverse containment; given
		$\tau_1$ and $\tau_2$ in $\tri_n$, we say that 
		$\tau_1 \leq \tau_2$ if $\tau_1$ contains $\tau_2$. Note 
		that the minimal elements of $\tri_n$ under this partial order 
		are the actual triangulations of $P_n$. Finally, a partial 
		triangulation of $P_{2n}$ is 
		\emph{centrally symmetric} if it is invariant under rotation by $\pi$.
		The set of all centrally symmetric elements, denoted $\overline{\tri}_{n}$,
		is a subposet of $\tri_{2n}$. 
	\end{defn}
	
	\begin{defn}
		The \emph{associahedron} $K_n$ is an $(n-2)$-dimensional convex
		polytope with face poset isomorphic to $\tri_{n+1}$. For example,
		$K_3$ is a line segment, $K_4$ is a pentagon, and
		$K_5$ is a polyhedron with $14$ vertices and $9$
		faces: six pentagons and three squares. See Figure~\ref{fig:assoc_poly}. 
		In general, faces of the associahedron may be expressed as 
		products of lower-dimensional associahedra.
	\end{defn}
	
	The associahedron was initially given a combinatorial description by 
	Tamari in 1951 before being rediscovered in a topological context by 
	Stasheff in the 1960s \cite{stasheff12}. Associahedra and their
	generalizations play important roles in several areas of mathematics, 
	including the studies of cluster algebras \cite{fomin07}, $A_\infty$-algebras
	and $A_\infty$-categories (such as the Fukaya category of a symplectic 
	manifold \cite{auroux14}), and moduli spaces of disks \cite{devadoss98}. 
	In physics, they also appear in the studies of open string theory and 
	scattering amplitudes \cite{arkani-hamed}.
	
	\begin{defn}
		The \emph{cyclohedron} $W_n$ is an $(n-1)$-dimensional convex 
		polytope with face poset isomorphic to $\overline{\tri}_{n}$. For example,
		$W_2$ is a line segment, $W_3$ is a hexagon, and $W_4$ is a 
		polyhedron with $20$ vertices and $12$ faces: four hexagons,
		four pentagons, and four squares. See Figure~\ref{fig:cyclo_poly}. 
		More broadly, faces of the cyclohedron are products of lower-dimensional
		associahedra and (at most one) cyclohedra.
	\end{defn}
	
	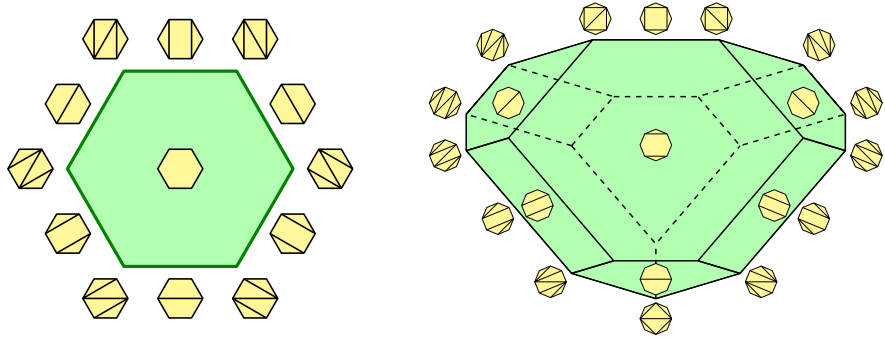
\begin{figure}
		\centering
		\begin{subfigure}{0.5\textwidth}
			\centering
			\begin{tikzpicture}
	\newcommand{\hexback}{
		\foreach \n in {1,...,6} {
			\coordinate (\n) at  (\n*60:0.2cm);
		}
		\draw[fill=yellow!50] (1)--(2)--(3)--(4)--(5)--(6)--cycle;
	}

	\foreach \i in {1,...,6} {
		\coordinate (v\i) at (\i*60:1cm);
	}
	\filldraw[GreenPoly] (v1)--(v2)--(v3)--(v4)--(v5)--(v6)--cycle;

	\begin{scope}
		\coordinate (top) at (0,0);
		\foreach \i in {1,...,6} \coordinate (a\i) at (\i*60+90:1.15cm);
		\foreach \i in {1,...,6} \coordinate (b\i) at (\i*60+120:1.3279cm);
	\end{scope}

	\begin{scope}[line join=bevel]
	
		\begin{scope}[shift={(top)}]
			\hexback
		\end{scope}
		
		\begin{scope}[shift={(a1)}]
			\hexback
			\draw (1) -- (4);
		\end{scope}
		
		\begin{scope}[shift={(a2)}]
			\hexback
			\draw (1) -- (3);
			\draw (4) -- (6);
		\end{scope}
		
		\begin{scope}[shift={(a3)}]
			\hexback
			\draw (3) -- (6);
		\end{scope}
		
		\begin{scope}[shift={(a4)}]
			\hexback
			\draw (2) -- (6);
			\draw (3) -- (5);
		\end{scope}
		
		\begin{scope}[shift={(a5)}]
			\hexback
			\draw (2) -- (5);
		\end{scope}
		
		\begin{scope}[shift={(a6)}]
			\hexback
			\draw (2) -- (4);
			\draw (5) -- (1);
		\end{scope}
		
		\begin{scope}[shift={(b1)}]
			\hexback
			\draw (3) -- (1) -- (4) -- (6);
		\end{scope}
		
		\begin{scope}[shift={(b2)}]
			\hexback
			\draw (1) -- (3) -- (6) -- (4);
		\end{scope}
		
		\begin{scope}[shift={(b3)}]
			\hexback
			\draw (5) -- (3) -- (6) -- (2);
		\end{scope}
		
		\begin{scope}[shift={(b4)}]
			\hexback
			\draw (3) -- (5) -- (2) -- (6);
		\end{scope}
		
		\begin{scope}[shift={(b5)}]
			\hexback
			\draw (1) -- (5) -- (2) -- (4);
		\end{scope}
	
		\begin{scope}[shift={(b6)}]
			\hexback
			\draw (5) -- (1) -- (4) -- (2);
		\end{scope}
	\end{scope}
	
\end{tikzpicture}
		\end{subfigure}%
		\begin{subfigure}{0.5\textwidth}
			\centering
			\begin{tikzpicture}
	
	\begin{scope}[line join = bevel]
	\begin{scope}[x={(3cm,-1.3cm)},y={(0cm,-2.2cm)},z={(-3cm,-1.3cm)}]
	
	
	\coordinate (a) at (1,1,1);
	\coordinate (b) at (1,-1,-1);
	\coordinate (c) at (-1,1,-1);
	\coordinate (d) at (-1,-1,1);
	
	\coordinate (aab) at ($3/5*(a) + 2/5*(b)$);	
	\coordinate (aac) at ($3/5*(a) + 2/5*(c)$);
	\coordinate (aad) at ($3/5*(a) + 2/5*(d)$);
	
	\coordinate (bba) at ($3/5*(b) + 2/5*(a)$);	
	\coordinate (bbc) at ($3/5*(b) + 2/5*(c)$);
	\coordinate (bbd) at ($3/5*(b) + 2/5*(d)$);
	
	\coordinate (cca) at ($3/5*(c) + 2/5*(a)$);	
	\coordinate (ccb) at ($3/5*(c) + 2/5*(b)$);
	\coordinate (ccd) at ($3/5*(c) + 2/5*(d)$);
	
	\coordinate (dda) at ($3/5*(d) + 2/5*(a)$);	
	\coordinate (ddb) at ($3/5*(d) + 2/5*(b)$);
	\coordinate (ddc) at ($3/5*(d) + 2/5*(c)$);
	
	\coordinate (aabc) at ($2/3*(aab) + 1/3*(aac)$);
	\coordinate (aabd) at ($2/3*(aab) + 1/3*(aad)$);
	
	\coordinate (aadb) at ($2/3*(aad) + 1/3*(aab)$);
	\coordinate (aadc) at ($2/3*(aad) + 1/3*(aac)$);
	
	\coordinate (bbac) at ($2/3*(bba) + 1/3*(bbc)$);
	\coordinate (bbad) at ($2/3*(bba) + 1/3*(bbd)$);
	
	\coordinate (bbca) at ($2/3*(bbc) + 1/3*(bba)$);
	\coordinate (bbcd) at ($2/3*(bbc) + 1/3*(bbd)$);
	
	\coordinate (ccba) at ($2/3*(ccb) + 1/3*(cca)$);
	\coordinate (ccbd) at ($2/3*(ccb) + 1/3*(ccd)$);
	
	\coordinate (ccda) at ($2/3*(ccd) + 1/3*(cca)$);
	\coordinate (ccdb) at ($2/3*(ccd) + 1/3*(ccb)$);
	
	\coordinate (ddca) at ($2/3*(ddc) + 1/3*(dda)$);
	\coordinate (ddcb) at ($2/3*(ddc) + 1/3*(ddb)$);
	
	\coordinate (ddab) at ($2/3*(dda) + 1/3*(ddb)$);
	\coordinate (ddac) at ($2/3*(dda) + 1/3*(ddc)$);
	
	
	\filldraw[fill=green!30,join=bevel] (aac) -- (aabc) -- (bbac) -- (bbca) --
	(bbcd) -- (bbd) -- (ddb) -- (ddcb) -- (ddca) -- (ddac) -- (aadc) -- cycle;
	

	\begin{scope}[thick]
		\draw (aabc) -- (aabd) -- (bbad) -- (bbac) -- cycle;
		\draw[dashed] (bbcd) -- (bbca) -- (ccba) -- (ccbd) -- cycle;
		\draw[dashed] (ccda) -- (ccdb) -- (ddcb) -- (ddca) -- cycle;
		\draw (ddab) -- (ddac) -- (aadc) -- (aadb) -- cycle;
		
		\draw (aabc) -- (aac) -- (aadc) -- (aadb) -- (aabd) -- cycle;
		\draw (bbcd) -- (bbd) -- (bbad) -- (bbac) -- (bbca) -- cycle;
		\draw[dashed] (ccdb) -- (ccbd);
		\draw[dashed] (ccda) -- (cca) -- (ccba);
		\draw (ddab) -- (ddb) -- (ddcb) -- (ddca) -- (ddac) -- cycle;
		
		\draw (bbd) -- (ddb);
		\draw[dashed] (aac) -- (cca);
	\end{scope}
	
	\end{scope}
	
	\newcommand{\octback}{
		\foreach \n in {1,...,8} {
			\coordinate (\n) at  (\n*45:0.3cm);
		}
		\draw[fill=yellow!50] (1)--(2)--(3)--(4)--(5)--(6)--(7)--(8)--cycle;
	}
	
	\begin{scope}[shift={(0,0.2)}]
		\octback
		\draw (1) -- (3);
		\draw (5) -- (7);
	\end{scope}
	
	\begin{scope}[shift={(0,-2.35)}]
		\octback
		\draw (8) -- (4);
	\end{scope}
	
	\begin{scope}[shift={(2.25,-0.95)}]
		\octback
		\draw (3) -- (8);
		\draw (4) -- (7);
	\end{scope}
	
	\begin{scope}[shift={(-2.25,-0.95)}]
		\octback
		\draw (1) -- (4);
		\draw (5) -- (8);
	\end{scope}
	
	\begin{scope}[shift={(2.8,1)}]
		\octback
		\draw (3) -- (7);
	\end{scope}
	
	\begin{scope}[shift={(-2.8,1)}]
		\octback
		\draw (1) -- (5);
	\end{scope}
	
	\begin{scope}[shift={(4,0)}]
	\octback
	\draw (3) -- (7);
	\draw (3) -- (8) -- (2);
	\draw (6) -- (4) -- (7);
	\end{scope}
	
	\begin{scope}[shift={(-4,0)}]
	\octback
	\draw (1) -- (5);
	\draw (1) -- (4) -- (2);
	\draw (5) -- (8) -- (6);
	\end{scope}
	
	\begin{scope}[shift={(4,1)}]
	\octback
	\draw (3) -- (7);
	\draw (8) -- (2) -- (7);
	\draw (3) -- (6) -- (4);
	\end{scope}
	
	\begin{scope}[shift={(-4,1)}]
	\octback
	\draw (1) -- (5);
	\draw (4) -- (2) -- (5);
	\draw (8) -- (6) -- (1);
	\end{scope}
	
	\begin{scope}[shift={(3.1,2.1)}]
	\octback
	\draw (3) -- (7);
	\draw (2) -- (7) -- (1);
	\draw (6) -- (3) -- (5);
	\end{scope}
	
	\begin{scope}[shift={(-3.1,2.1)}]
	\octback
	\draw (1) -- (5);
	\draw (2) -- (5) -- (3);
	\draw (6) -- (1) -- (7);
	\end{scope}

	\begin{scope}[shift={(0,2.6)}]
		\octback
		\draw (1) -- (3) -- (5) -- (7) -- cycle;
	\end{scope}
	
	\begin{scope}[shift={(-1.15,2.6)}]
		\octback
		\draw (1) -- (3) -- (5) -- (7) -- cycle;
		\draw (1) -- (5);
	\end{scope}
	
	\begin{scope}[shift={(1.15,2.6)}]
		\octback
		\draw (1) -- (3) -- (5) -- (7) -- cycle;
		\draw (3) -- (7);
	\end{scope}
	
	\begin{scope}[shift={(2,-2.4)}]
	\octback
	\draw (4) -- (8);
	\draw (3) -- (8) -- (2);
	\draw (6) -- (4) -- (7);
	\end{scope}
	
	\begin{scope}[shift={(-2,-2.4)}]
	\octback
	\draw (4) -- (8);
	\draw (1) -- (4) -- (2);
	\draw (5) -- (8) -- (6);
	\end{scope}
	
	\begin{scope}[shift={(3,-1.2)}]
	\octback
	\draw (3) -- (8) -- (2);
	\draw (6) -- (4) -- (7);
	\end{scope}
	
	\begin{scope}[shift={(-3,-1.2)}]
	\octback
	\draw (1) -- (4) -- (2);
	\draw (5) -- (8) -- (6);
	\end{scope}
	
	\begin{scope}[shift={(0,-3.1)}]
	\octback
	\draw (2) -- (4) -- (6) -- (8) -- cycle;
	\draw (4) -- (8);
	\end{scope}
	
	\end{scope}
	
\end{tikzpicture}
		\end{subfigure}
		\caption{The cyclohedra $W_3$ and $W_4$, with some of the
			front-facing cells labeled}
		\label{fig:cyclo_poly}
	\end{figure}
	
	Among the many generalizations of the associahedron, cyclohedra
	are perhaps the closest relative. They first appeared in the context 
	of knot invariants, where they were given a description by Bott and 
	Taubes \cite{bott94}. Realizations of the cyclohedron as a convex polytope
	later followed in the work of Markl \cite{markl99} and Simion \cite{simion03}.
	
	In the planar tree poset, the face posets of associahedra and cyclohedra appear
	as the lower sets of particular types of trees.
	
	\begin{defn}\label{def:stars}
		If $\Gamma$ is a planar $n$-tree with a unique interior vertex, 
		then $\Gamma$ is a \emph{star}. Since each leaf must be marked, there 
		are two types of stars; $\Gamma$ is a \emph{marked star}
		if the unique interior vertex is marked and
		an \emph{unmarked star} otherwise.
		Let $\ustar_n$ denote the unmarked star with $n$ marked 
		points, labeled $1,\ldots,n$ in counter-clockwise order
		around the unique unmarked interior vertex. Let
		$\mstar_n$ denote the marked star with $n$ marked points:
		a single interior vertex labeled $n$, and $n-1$ leaves
		labeled $1,\ldots,n-1$ in counter-clockwise order
		around the unique marked interior vertex. See Figure~\ref{fig:stars}.
	\end{defn}
	
	\begin{rem}\label{rem:star-drawing}
		The trees which lie below $\ustar_{n+1}$ or $\mstar_{n+1}$
		in the partial order admit a canonical planar embedding.
		Let $\Gamma$ be a planar $n$-tree in the lower set $\low(\ustar_{n+1})$.
		Then each marked point of $\Gamma$ is a leaf, so one can
		see that there is a natural representation of $\Gamma$ in $\mathbb{C}$
		where each marked point labeled $i$ is sent to $(z_i+z_{i+1})/2$ 
		(i.e. the midpoint of a side of $P_{n+1}$) and
		the entire tree is contained within the unit disk. 
		Similarly, if $\Gamma$ is a planar $n$-tree in $\low(\mstar_{n+1})$,
		then there is a canonical planar embedding which sends the $n$ 
		(marked) leaves to the $n$ edges of $P_{n+1}$ as above and which sends the
		unique marked interior vertex to the origin. In both cases above, the
		embedding is unique up to isotopy fixing the marked points. 
	\end{rem}
	
	\begin{figure}
		\centering
		\begin{tikzpicture}
	
	\begin{scope}[shift={(-1.5,0)}]
		\foreach \i in {1,...,6} {
			\node [dot,color=blue!70] (\i) at (60*\i-60:0.8cm) {};
			\node () at (60*\i-60:1cm) {\tiny $\i$};
		}
		\node [dot,fill=white] (a) at (0,0) {};
		\begin{scope}[thick]
			\draw (a) -- (1);
			\draw (a) -- (2);
			\draw (a) -- (3);
			\draw (a) -- (4);
			\draw (a) -- (5);
			\draw (a) -- (6);
		\end{scope}
	\end{scope}
	
	\begin{scope}[shift={(1.5,0)}]
		\foreach \i in {1,...,5} {
			\node [dot,color=blue!70] (\i) at (72*\i-54:0.8cm) {};
			\node () at (72*\i-54:1cm) {\tiny $\i$};
		}
		\node [dot,color=blue!70] (a) at (0,0) {};
		\node () at (54:0.2cm) {\tiny $6$};
		\begin{scope}[thick]
		\draw (a) -- (1);
		\draw (a) -- (2);
		\draw (a) -- (3);
		\draw (a) -- (4);
		\draw (a) -- (5);
		\end{scope}
	\end{scope}

\end{tikzpicture}
		\caption{The unmarked star $\ustar_6$ and the marked star $\mstar_6$}
		\label{fig:stars}
	\end{figure}
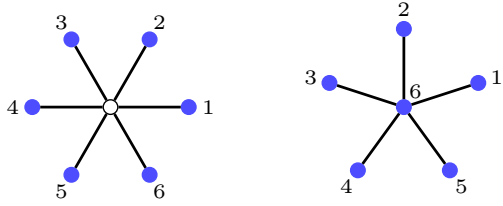
	
	Equivalent versions of the following two lemmas have appeared previously
	in the literature, although we prove them here for the sake of completeness.
	
	\begin{lem}[\cite{devadoss98}]\label{lem:star-assoc}
		The face poset of the associahedron $K_n$ is isomorphic to 
		$\low(\ustar_{n+1})$.
	\end{lem}
	
	\begin{proof}
		Let $\tau$ be a partial triangulation in $\tri_{n+1}$; we
		will construct a corresponding ``dual'' planar tree in 
		$\low(\ustar_{n+1})$.
		For each $k\in \{1,\ldots,n+1\}$, place a marked vertex 
		labeled $k$ at $(z_k + z_{k+1})/2$. For each polygonal region of $P_{n+1}$
		formed by $\tau$, place an unmarked vertex at its barycenter. 
		Add an edge from each unmarked vertex to the marked vertices labeling
		the corners of the corresponding polygonal region and to the
		unmarked vertices labeling adjacent polygonal regions.
		The resulting connected planar graph must be a tree since cutting
		$P_{n+1}$ along any diagonal in $\tau$ splits the polygon
		into two pieces, and this corresponds to deletion of any
		edge disconnecting the graph.
		Moreover, this process is reversible since each planar tree in $\low(\ustar_{n+1})$ can be
		drawn on $P_{n+1}$ by Remark~\ref{rem:star-drawing}, from which
		we may recover the corresponding partial triangulation. 
		Thus, we have defined a bijection $\phi \colon \tri_{n+1} \to \planar_{n+1}$.
		See Figure~\ref{fig:ustar_map} for an example.
		Removing a diagonal in $\tau$ corresponds exactly to contracting an 
		edgein $\phi(\tau)$, so $\phi$ is a poset isomorphism.
	\end{proof}
	
	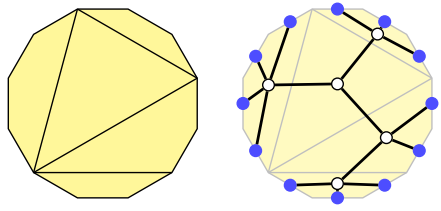
\begin{figure}
		\centering
		\begin{tikzpicture}

	\begin{scope}[shift={(-1.2,0)}]
		\foreach \i in {1,...,12} {
			\coordinate (v\i) at (\i*30+15:1cm);
		}
		\filldraw[fill=yellow!50] (v1)--(v2)--(v3)--(v4)--(v5)--(v6)--
		(v7)--(v8)--(v9)--(v10)--(v11)--(v12)--cycle;
	\end{scope}
	
	\begin{scope}[shift={(-1.2,0)}]
		\draw (v3) -- (v12);
		\draw (v7) -- (v12);
		\draw (v3) -- (v7);
		\draw (v7) -- (v10);
	\end{scope}

	\begin{scope}[shift={(1.2,0)}]
	
	\foreach \i in {1,...,12} {
		\coordinate (v\i) at (\i*30+15:1cm);
	}
	\filldraw[color=gray!50,fill=yellow!30] (v1)--(v2)--(v3)--(v4)--(v5)--(v6)--
	(v7)--(v8)--(v9)--(v10)--(v11)--(v12)--cycle;
	
	\begin{scope}[color=gray!50]
		\draw (v3) -- (v12);
		\draw (v7) -- (v12);
		\draw (v3) -- (v7);
		\draw (v7) -- (v10);
	\end{scope}
	
	\node [littledot,fill=white] (u1) at (0,0.2) {};
	\node [littledot,fill=white] (u2) at (0.5,-0.35) {};
	\node [littledot,fill=white] (u3) at (165:0.73cm) {};
	\node [littledot,fill=white] (u4) at (60:0.82cm) {};
	\node [littledot,fill=white] (u5) at (270:0.82cm) {};
	\foreach \i in {1,...,12} {
		\node [littledot,color=blue!70] (v\i) at (\i*30:0.966cm) {};
	}

	\begin{scope}[thick]
		\draw (v1) -- (u4);
		\draw (v2) -- (u4);
		\draw (v3) -- (u4);
		\draw (v4) -- (u3);
		\draw (v5) -- (u3);
		\draw (v6) -- (u3);
		\draw (v7) -- (u3);
		\draw (v8) -- (u5);
		\draw (v9) -- (u5);
		\draw (v10) -- (u5);
		\draw (v11) -- (u2);
		\draw (v12) -- (u2);
		\draw (u1) -- (u2);
		\draw (u1) -- (u3);
		\draw (u2) -- (u5);
		\draw (u4) -- (u1);
	\end{scope}
	
	\end{scope}
	
\end{tikzpicture}
		\caption{A partial triangulation in $\tri_{12}$ and the 
			associated planar tree in $\low(\ustar_{12})$}
		\label{fig:ustar_map}
	\end{figure}
	
	\begin{lem}[\cite{devadoss03}]\label{lem:star-cyclo}
		The face poset of the cyclohedron $W_n$ is isomorphic to $\low(\mstar_{n+1})$.
	\end{lem}
	
	\begin{proof}
		Let $\overline{\low(\ustar_{2n})}$ denote the subposet of trees in
		$\low(\ustar_{2n})$ which are invariant	under rotation by $\pi$
		and note that this subposet is isomorphic to $\overline{\tri}_n$ 
		by Lemma~\ref{lem:star-assoc}. Each $\Gamma \in \overline{\low(\ustar_{2n})}$
		can be canonically represented in the polygon $P_{2n}$, where the center
		point is either an unmarked vertex of $\Gamma$ or the midpoint of 
		an edge in $\Gamma$. Either way, delete the center point, take the 
		quotient of the tree by a $\pi$ rotation, and replace the missing point
		with a marked vertex to obtain a new graph $\Gamma^\prime$.
		The marked vertices labeled $k$ and $n+k$ in $\Gamma$ are
		identified to a single marked vertex labeled $k$ in $\Gamma^\prime$.
		Label the new marked interior vertex by $n+1$ and note that $\Gamma^\prime$ 
		is a planar tree with $n+1$ marked points. Moreover, we know that
		$\Gamma$ contained a subtree which could be contracted to obtain
		$\ustar_{2n}$, so by contracting the corresponding subtree of
		$\Gamma^\prime$ (which now contains the marked point $n+1$),
		we obtain $\mstar_{n+1}$, and thus $\Gamma^\prime \in \low(\mstar_{n+1})$.
		See Figure~\ref{fig:mstar_map} for an example.
		
		Define $\psi\colon \overline{\low(\ustar_{2n})} \to \low(\mstar_{n+1})$
		by $\psi(\Gamma) = \Gamma^\prime$.  The quotient by $\pi$ respects edge-contraction (so $\psi$ is an
		order embedding) and can be undone to obtain a centrally
		symmetric planar tree (so $\psi$ is a bijection). Therefore, 
		$\psi$ is an isomorphism.
	\end{proof}
	
	These lemmas provide the necessary tools to examine the lower set 
	$\low(\Gamma)$ for an arbitrary tree $\Gamma\in \planar_n$.
	
	\begin{defn}\label{def:nbhd}
		Let $\Gamma$ be a planar $n$-tree and let $\Gamma_1$ be a
		subtree of $\Gamma$. Then the (closed) \emph{neighborhood} of
		$\Gamma_1$ in $\Gamma$, denoted $\nbhd(\Gamma_1)$, is the smallest
		subtree containing $\Gamma_1$ and all of its incident
		edges. In particular, if $v$ is a vertex of $\Gamma$ with
		valence $k$, then $\nbhd(v)$ is a star. With the appropriate
		labeling, $\nbhd(v)$ is isomorphic to either
		$\ustar_k$ (if $v$ is unmarked) or $\mstar_k$ (if $v$ is marked).	
	\end{defn}

	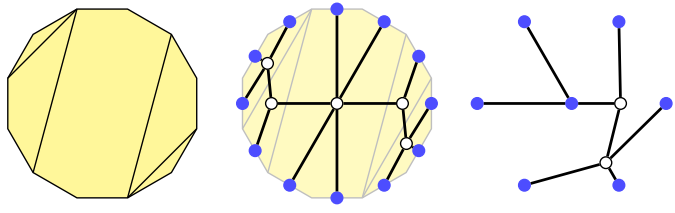
\begin{figure}
		\centering
		\begin{tikzpicture}

	
	\begin{scope}[shift={(-1.2,0)}]
	\foreach \i in {1,...,12} {
		\coordinate (v\i) at (\i*30+15:1cm);
	}
	\filldraw[fill=yellow!50] (v1)--(v2)--(v3)--(v4)--(v5)--(v6)--
	(v7)--(v8)--(v9)--(v10)--(v11)--(v12)--cycle;
	
	\draw (v3) -- (v7);
	\draw (v3) -- (v5);
	\draw (v9) -- (v1);
	\draw (v9) -- (v11);
	\end{scope}

	\begin{scope}[shift={(1.2,0)}]
	\foreach \i in {1,...,12} {
		\coordinate (v\i) at (\i*30+15:1cm);
	}
	\filldraw[color=gray!50,fill=yellow!30] (v1)--(v2)--(v3)--(v4)--(v5)--(v6)--
	(v7)--(v8)--(v9)--(v10)--(v11)--(v12)--cycle;
	
	\begin{scope}[color=gray!50]
		\draw (v3) -- (v7);
		\draw (v3) -- (v6);
		\draw (v9) -- (v1);
		\draw (v9) -- (v12);
	\end{scope}
	
	\node [littledot,fill=white] (center) at (0,0) {};
	\node [littledot,fill=white] (u2) at (0:0.67) {};
	\node [littledot,fill=white] (u3) at (-30:0.82cm) {};
	\node [littledot,fill=white] (u5) at (180:0.67cm) {};
	\node [littledot,fill=white] (u6) at (150:0.82cm) {};
	\foreach \i in {1,...,12} {
		\node [littledot,color=blue!70] (v\i) at (\i*30:0.966cm) {};
	}
	
	\begin{scope}[thick]
		\draw (center) -- (u2) -- (u3);
		\draw (center) -- (u5) -- (u6);
		\draw (center) -- (v2);
		\draw (center) -- (v3);
		\draw (u2) -- (v1);
		\draw (u3) -- (v10);
		\draw (u3) -- (v11);
		\draw (u3) -- (v12);
		\draw (center) -- (v8);
		\draw (center) -- (v9);
		\draw (u5) -- (v7);
		\draw (u6) -- (v4);
		\draw (u6) -- (v5);
		\draw (u6) -- (v6);
	\end{scope}
	
	\begin{scope}[shift={(2.4,0)}]
	
	\node [littledot,color=blue!70] (center) at (0,0) {};
	\node [littledot,fill=white] (u2) at (0:0.5) {};
	\node [littledot,fill=white] (u3) at (300:0.7cm) {};
	\foreach \i in {1,...,6} {
		\node [littledot,color=blue!70] (v\i) at (\i*60:0.966cm) {};
	}
	
	\begin{scope}[thick]
		\draw (center) -- (u2) -- (u3);
		\draw (center) -- (v2);
		\draw (center) -- (v3);
		\draw (u2) -- (v1);
		\draw (u3) -- (v4);
		\draw (u3) -- (v5);
		\draw (u3) -- (v6);
	\end{scope}
	
	\end{scope}
	
	\end{scope}
\end{tikzpicture}
		\caption{A centrally symmetric partial triangulation in $\tri_{12}$ and
			its associated planar trees in $\low(\ustar_{12})$ and 
			$\low(\mstar_7)$.}
		\label{fig:mstar_map}
	\end{figure}
	
	\begin{thm}\label{thm:poset-product}
		Let $\Gamma$ be a planar $n$-tree with unmarked interior
		vertices $u_1,\ldots,u_k$ and marked interior vertices
		$v_1,\ldots,v_\ell$. Then 
		\[\low(\Gamma) \cong \prod_{1\leq i \leq k} \low\left(\ustar_{\val(u_i)}\right)
		\times \prod_{1\leq j \leq \ell} \low\left(\mstar_{\val(v_j)}\right)\]
		and therefore $\low(\Gamma)$ is the face poset of a convex
		polytope which can be expressed as a product of associahedra and
		cyclohedra. 
	\end{thm}
	
	\begin{proof}
		Let $\Gamma$ be a planar $n$-tree, let $u$ be an 
		unmarked interior vertex of $\Gamma$, and fix an identification between
		$\nbhd(u)$ and $\ustar_{\val(u)}$ as described in Definition~\ref{def:nbhd}.
		For each $\Gamma^\prime \in \low(\Gamma)$, we know that some subtree 
		$\Gamma^\prime_u$ of $\Gamma^\prime$ contracts to $u$ and so 
		$\nbhd(\Gamma^\prime_u)$ contracts to $\nbhd(u)$. Abusing notation
		slightly, we denote this by writing $\nbhd(\Gamma^\prime_u) \in \low(\ustar_{\val(u)})$.
		Similarly, we can see that each interior marked
		vertex $v$ of $\Gamma$ has a corresponding subtree $\Gamma^\prime_v$
		of $\Gamma^\prime$ such that $\nbhd(\Gamma^\prime_v) \in \low(\mstar_{\val(v)})$.
		
		Now, let $u_1,\ldots,u_k$ and $v_1,\ldots,v_\ell$ be the
		unmarked and marked interior vertices of $\Gamma$, respectively.
		Identify each $\nbhd(u_i)$ with $\ustar_{\val(u_i)}$ and
		each $\nbhd(v_i)$ with $\mstar_{\val(v_j)}$ as before and
		define the function 
		\[
		\psi\colon \low(\Gamma) \to 
		\prod_{1\leq i \leq k} \low\left(\ustar_{\val(u_i)}\right)
		\times \prod_{1\leq j \leq \ell} \low\left(\mstar_{\val(v_j)}\right)
		\]
		by declaring $\psi(\Gamma^\prime) = (\nbhd(\Gamma^\prime_{u_1}),\ldots,\nbhd(\Gamma^\prime_{u_k}),
		\nbhd(\Gamma^\prime_{v_1}),\ldots,\nbhd(\Gamma^\prime_{v_\ell}))$. We
		will show that $\psi$ is a poset isomorphism.

		Given any $u_i$ and an element $\Lambda_i \in \low(\ustar_{\val(u_i)})$,
		we can use the identification fixed previously to replace $\nbhd(u_i)$ with
		$\Lambda_i$, and the resulting planar $n$-tree will be an element of 
		$\low(\Gamma)$ since, by construction, it can be contracted to $\Gamma$. 
		Similarly, we can replace $\nbhd(v_j)$ with any element of
		$\low(\mstar_{\val(v_j)})$ to obtain an element of $\low(\Gamma)$.
		Moreover, these replacements do not interfere with one another, so
		any $(k+\ell)$-tuple in the codomain of $\psi$ determines a unique
		$n$-tree $\Gamma^\prime$ in $\low(\Gamma)$. Finally, note that that
		$\psi(\Gamma^\prime)$ returns our original $(k+\ell)$-tuple by
		construction and therefore $\psi$ is a bijection.
		
		Next, let $\Gamma^\prime$ and $\Gamma^{\pprime}$ be elements of 
		$\low(\Gamma)$. We know by the first part of Corollary~\ref{cor:subforests}
		that $\Gamma^{\prime} \leq \Gamma^{\pprime}$
		if and only if $F(\Gamma^{\pprime},\Gamma)$ is a subforest of 
		$F(\Gamma^\prime,\Gamma)$. This is 
		equivalent to saying that $\Gamma^\pprime_{u_i} \leq \Gamma^\prime_{u_i}$
		and $\Gamma^\pprime_{v_j} \leq \Gamma^\prime_{v_j}$ for each $i$ and $j$,
		which is then equivalent to saying that 
		$\nbhd(\Gamma^\pprime_{u_i}) \leq \nbhd(\Gamma^\prime_{u_i})$ in $\ustar_{\val(u_i)}$
		and $\nbhd(\Gamma^\pprime_{v_j}) \leq \nbhd(\Gamma^\prime_{v_j})$ in $\mstar_{\val(v_j)}$
		for each $i$ and $j$.
		We have thus shown that $\Gamma^\prime \leq \Gamma^\pprime$ if
		and only if $\psi(\Gamma^\prime) \leq \psi(\Gamma^\pprime)$, so
		$\psi$ is an order embedding and therefore an isomorphism.
	\end{proof}
	
	\begin{example}\label{ex:convex-polytope}
		If $\Gamma$ is the planar $6$-tree depicted in Figure~\ref{fig:n-tree},
		then the lower set $\low(\Gamma)$ is isomorphic to the direct product
		$\low(\ustar_4) \times \low(\mstar_3)$, which is isomorphic to
		the face poset for the product of $K_3\times W_2$, i.e. a hexagonal prism.
	\end{example}
	
	\begin{rem}\label{rem:planar-complex}
		As Theorem~\ref{thm:poset-product} suggests, $\planar_n$ is the
		face poset of a regular cell complex in which each cell is a convex
		polytope. As depicted in Figure~\ref{fig:p3_poset}, $\planar_3$
		is the face poset of a $1$-dimensional complex with 
		3 edges and 2 vertices (i.e. a theta graph). 
		Meanwhile, $\planar_4$ is the
		face poset of a $2$-dimensional complex consisting of 20 $2$-cells 
		(eight hexagons and twelve squares), 30 edges, and 12 vertices.
	\end{rem}
	
	\section{Tree Complexes}
	\label{sec:tree-complexes}
	
	In this section, we consider planar tree embeddings up to 
	relative isotopy fixing the marked vertices pointwise.
	This infinite set of trees admits a similar partial order 
	by contraction and forms the face poset for a contractible
	cell complex.
	
	\begin{defn}
		Fix a set of points $P = \{z_1,\ldots,z_n\}$ in $\mathbb{C}$,
		where $n\geq 3$. A \emph{planted $n$-tree}
		is a relative isotopy class of embeddings $\phi\colon \Gamma\to\mathbb{C}$
		where $\Gamma$ is an $n$-tree such that each marked point $v_i$ in $\Gamma$
		is sent to $z_i$ in $P$. Following Definition~\ref{def:contraction-planar},
		we say that if $\phi_1\colon\Gamma_1\to\mathbb{C}$ and 
		$\phi_2\colon\Gamma_2\to\mathbb{C}$ are planted $n$-trees and 
		$f\colon \Gamma_1\to\Gamma_2$ is a contraction such that 
		$\phi_1$ and $\phi_2\circ f$ are isotopic relative to the
		marked points, then $f$ is a \emph{planted contraction}.
		If such a contraction exists, we write $\Gamma_1\leq\Gamma_2$
		and observe that this determines a partial order on the
		set of all planted $n$-trees. We refer to this partially ordered 
		set as the \emph{planted tree poset} and denote it $\planted_n$.
	\end{defn}
	
	It is worth noting that the combinatorial structure of the planted
	tree poset does not depend on our choice of $P$. More explicitly,
	if $\planted_n(P)$ and $\planted_n(P^\prime)$ are the posets determined
	by two sets $P$ and $P^\prime$ of $n$ points in $\mathbb{C}$, then
	any homeomorphism $\mathbb{C}\to\mathbb{C}$ which takes $P$ to $P^\prime$
	induces an isomorphism $\planted_n(P) \to \planted_n(P^\prime)$.
	
	\begin{figure}
		\centering
		\begin{tikzpicture}

\tikzstyle{bigdot}=[dot,inner sep=3pt]
\tikzstyle{disk}=[very thick,shape=circle,draw,color=black,fill=yellow!30, inner sep = 1.1cm]

	\newcommand{\drawtree}{
	\foreach \i in {1,2,3} {
		\node[bigdot, color = blue!70] (v\i) at (\i*120-30:0.7cm) {};
	}
	}
	
	\begin{scope}{color = black}
		\node[dot] (x) at (0,0) {};
		
		\foreach \i in {1,2,3} {
			\newcommand{\rr}{120*\i}
			\node[dot] (x\i) at (\i*120+90:8cm) {};
			\draw[GreenLine, line width = 1.5mm] (x) -- (x\i);

			\filldraw[fill = yellow!40] (0,0) circle (1.3cm);
			\drawtree
			\node[bigdot, fill = white] (u) at (0,0) {};
			\draw[thick] (u) -- (v1);
			\draw[thick] (u) -- (v2);
			\draw[thick] (u) -- (v3);
			
			\begin{scope}[shift={(\i*120+90:4cm)}]
				\filldraw[fill = yellow!40] (0,0) circle (1.3cm);
				\draw[thick] (-30+\rr:0.7cm) -- (90+\rr:0.7cm) -- (-150+\rr:0.7cm); 
				\drawtree
			\end{scope}
			
			\begin{scope}[shift={(\i*120+90:8cm)}]
				\foreach \j in {1,2} {
					\node[dot] (x\i\j) at ({120*(\i+\j)-90}:8cm) {};
					\draw[GreenLine, line width = 1.5mm] (x\i\j) -- (x\i);
				}
			
				\begin{scope}[shift={({120*(\i+1)-90}:4cm)}]
					\filldraw[fill = yellow!40] (0,0) circle (1.3cm);
					\draw[thick] (-150+\rr:0.7cm) 
					to[out=90+\rr, in=-180+\rr] (90+\rr:1cm)
					to[out=\rr, in=60+\rr] (-30+\rr:0.7cm) -- (90+\rr:0.7cm);
					\drawtree
				\end{scope}
				
				\begin{scope}[shift={({120*(\i+1)-90}:8cm)}]
					\filldraw[fill = yellow!40] (0,0) circle (1.3cm);
					\node[bigdot, fill = white] (u) at (-90+\rr:1cm) {};
					\draw[thick] (-30+\rr:0.7cm) -- (u);
					\draw[thick] (90+\rr:0.7cm) -- (u);
					\draw[thick] (-150+\rr:0.7cm) 
					to[out=90+\rr, in=-180+\rr] (90+\rr:1cm)
					to[out=\rr, in=60+\rr] (-30+\rr:1cm)
					to[out=-120+\rr, in=\rr] (u);
					\drawtree
				\end{scope}
				
				\begin{scope}[shift={({120*(\i+2)-90}:4cm)}]
					\filldraw[fill = yellow!40] (0,0) circle (1.3cm);
					\draw[thick] (-30+\rr:0.7cm) 
					to[out=90+\rr, in=\rr] (90+\rr:1cm)
					to[out=180+\rr, in=120+\rr] (-150+\rr:0.7cm) -- (90+\rr:0.7cm);
					\drawtree
				\end{scope}
				
				\begin{scope}[shift={({120*(\i+2)-90}:8cm)}]
				\filldraw[fill = yellow!40] (0,0) circle (1.3cm);
				\node[bigdot, fill = white] (u) at (-90+\rr:1cm) {};
				\draw[thick] (-150+\rr:0.7cm) -- (u);
				\draw[thick] (90+\rr:0.7cm) -- (u);
				\draw[thick] (-30+\rr:0.7cm) 
				to[out=90+\rr, in=\rr] (90+\rr:1cm)
				to[out=180+\rr, in=120+\rr] (-150+\rr:1cm)
				to[out=-60+\rr, in=180+\rr] (u);
				\drawtree
				\end{scope}
				
				\filldraw[fill = yellow!40] (0,0) circle (1.3cm);
				\node[bigdot, fill = white] (u) at (90+\rr:1.1cm) {};
				\draw[thick] (-30+\rr:0.7cm) to[out=60+\rr, in=\rr] (u);
				\draw[thick] (-150+\rr:0.7cm) to[out=120+\rr, in=180+\rr] (u);
				\draw[thick] (u) -- (90+\rr:0.7cm);
				\drawtree
			\end{scope}
		}
		
	\end{scope}

\end{tikzpicture}
		\caption{A piece of the simplicial tree complex $\mathcal{S}_3$}
		\label{fig:simplicial-tree-complex}
	\end{figure}
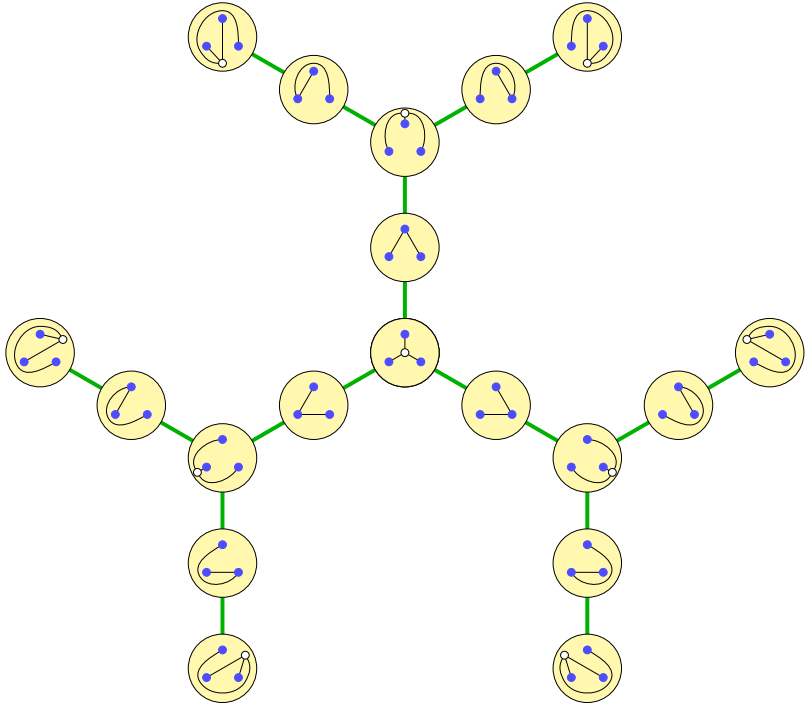
	
	\begin{defn}
		The \emph{simplicial tree complex} $\mathcal{S}_n$ is the order complex of the
		planted tree poset $\planted_n$. For example, $\mathcal{S}_3$ is isomorphic to the 
		infinite bipartite tree $T_{2,3}$, depicted in Figure~\ref{fig:simplicial-tree-complex}.
	\end{defn}
	
	This complex appeared in recent work
	of Belk, Lanier, Margalit, and Winarski on complex dynamics \cite{belk}.
	In this setting, each postcritically finite complex polynomial 
	(i.e. one in which the critical points have finite forward orbits) 
	has an associated \emph{Hubbard tree}. The simplicial tree complex 
	then acts as a useful tool for studying transformations
	of polynomials via their Hubbard trees. Using a result of Penner 
	\cite{penner96}, the authors in \cite{belk} describe an 
	embedding of the simplicial tree complex as a spine for the 
	Teichm\"uller space of the $(n+1)$-punctured sphere.
	As a consequence, they conclude that the simplicial tree complex is contractible.
	The focus of this article is to examine the combinatorial structure of the simplicial 
	tree complex and introduce a simpler polyhedral cell structure.
	
	Any self-homeomorphism of the plane which fixes $P$ pointwise sends
	each planted $n$-tree to another planted $n$-tree, and the only homeomorphisms
	which send a tree to itself are those which are isotopic to the identity. 
	Moreover, if $\phi_1\colon\Gamma_1 \to \mathbb{C}$
	and $\phi_2\colon\Gamma_2 \to \mathbb{C}$ are planted $n$-trees, then
	$\phi_1$ and $\phi_2$ are freely isotopic if and only if there is a homeomorphism
	$g\colon \mathbb{C}\to \mathbb{C}$ such that $\phi_1$ and $g\circ \phi_2$ 
	are isotopic relative to the marked points.
	In other words, if we let $\mathbb{C}_P$ denote the $n$-punctured plane 
	$\mathbb{C}-P$, then the pure mapping class group $\PMod(\mathbb{C}_P)$
	acts freely on $\planted_n$, and the orbits under this action correspond exactly
	to planar $n$-trees. This action is equivariant with respect to the
	partial order on the planted tree poset, which is to say that for any
	$g\in \PMod(\mathbb{C}_P)$, 	$\Gamma_1 \leq \Gamma_2$
	in $\planted_n$ if and only if $g\Gamma_1 \leq g\Gamma_2$.
	
	\begin{defn}\label{def:quotient-map}		
		Define the order-preserving surjective map 
		$p\colon \planted_n \to \planar_n$ by sending each planted $n$-tree 
		$\phi\colon \Gamma\to \mathbb{C}$ to its free isotopy class in $\planar_n$,
		and observe that the preimages under this map are the orbits of the action by
		$\PMod(\mathbb{C}_P)$. In particular, the partial order on the planar tree poset
		matches the partial order on the orbits: $p(\Gamma_1) \leq p(\Gamma_2)$
		in $\planar_n$ if and only if $g\Gamma_1 \leq \Gamma_2$ in $\planted_n$
		for some $g\in \PMod(\mathbb{C}_P)$. 
	\end{defn}

	Our first step is to strengthen an observation from the definition above.

	\begin{lem}\label{lem:tree-lifting}
		Let $\Gamma_1$ and $\Gamma_2$ be planted $n$-trees. Then 
		$p(\Gamma_1) \leq p(\Gamma_2)$ in $\planar_n$ if and only if 
		there is a unique $g\in \PMod(\mathbb{C}_P)$ such that 
		$g\Gamma_1 \leq \Gamma_2$ in $\planted_n$.
	\end{lem}

	\begin{proof}
		Let $\Gamma_1$ and $\Gamma_2$ be planted $n$-trees with
		planar embeddings $\phi_1\colon\Gamma_1 \to \mathbb{C}$ and
		$\phi_2\colon\Gamma_2 \to \mathbb{C}$ respectively.
		We can see from Definition~\ref{def:quotient-map} that 
		$p(\Gamma_1) \leq p(\Gamma_2)$ in $\planar_n$ if and only if 
		$g\Gamma_1 \leq \Gamma_2$ in $\planted_n$ for \emph{some}
		$g$ in the pure mapping class group $\PMod(\mathbb{C}_P)$;
		all that remains is to show that $g$ must be unique. 
		Suppose $p(\Gamma_1) \leq p(\Gamma_2)$ and that
		$g \Gamma_1 \leq \Gamma_2$ and $h \Gamma_1 \leq \Gamma_2$
		for some $g,h\in \PMod(\mathbb{C}_P)$. By Lemma~\ref{lem:unique-contraction},
		both of these inequalities must be realized by the same contraction
		$f\colon \Gamma_1\to\Gamma_2$. We then know that $g\circ \phi_1$
		is relatively isotopic to $\phi_2\circ f$ since $g \Gamma_1 \leq \Gamma_2$
		and $h\circ \phi_1$
		is relatively isotopic to $\phi_2\circ f$ since $h \Gamma_1 \leq \Gamma_2$.
		Therefore, $g\circ \phi_1$ is isotopic to $h\circ \phi_2$ relative to
		the marked points, and since $\PMod(\mathbb{C}_P)$ acts freely on
		the planted tree poset, we may conclude that $g=h$.
	\end{proof}

	The preceding lemma implies that $p$ restricts to injective maps on
	the lower set $\low(\Gamma)$ and the upper set $\upper(\Gamma)$ 
	(recall Definition~\ref{def:upper-lower}).

	\begin{lem}\label{lem:planted-planar-iso}
		Let $\Gamma$ be a planted $n$-tree. Then $\low(\Gamma)$ is isomorphic to 
		$\low(p(\Gamma))$ and $\upper(\Gamma)$ is isomorphic to $\upper(p(\Gamma))$.	
	\end{lem}
	
	\begin{proof}
		Let $\Gamma$ be a planted $n$-tree. By Lemma~\ref{lem:tree-lifting}, we know
		that every element of $\low(p(\Gamma))$ has a unique preimage under $p$
		which lies in $\low(\Gamma)$. In particular, this lemma tells us that 
		restricting $p$ to a map $\low(\Gamma)\to \low(p(\Gamma))$ yields a
		surjective order embedding, i.e. a poset isomorphism. By noting that
		$g\Gamma_1 \leq \Gamma_2$ if and only if $\Gamma_1 \leq g^{-1}\Gamma_2$,
		we obtain the analogous result for $\upper(\Gamma)$ as well.
	\end{proof}

	Now that we have established the connection between planted and planar trees,
	we are ready to introduce a simplified cell structure for the simplicial 
	tree complex.

	\begin{thm}[Theorem~\ref{thma:polyhedral-tree-complex}]\label{thm:polyhedral-tree-complex}
		The planted tree poset $\planted_n$ is the face poset of a regular
		CW-complex $\mathcal{P}_n$ which we call the \emph{polyhedral tree complex}, 
		in which each planted $n$-tree $\Gamma$ labels a cell which is isomorphic to
		the convex polytope labeled by the planar $n$-tree $p(\Gamma)$,
		and gluing relations are given by the partial order in $\planted_n$.
		In particular, the top-dimensional cells are products of cyclohedra.
		Consequently, the simplicial tree complex is the barycentric subdivision
		of the polyhedral tree complex.
	\end{thm}

	\begin{proof}
		The first claim follows from the definitions and Lemma~\ref{lem:planted-planar-iso}.
		The fact that top-dimensional cells can be expressed as products of cyclohedra
		is a consequence of Theorem~\ref{thm:poset-product}. 
		The final claim follows from the relationship between
		order complexes and their face posets.
	\end{proof}
	
	Since the simplicial tree complex is contractible \cite{belk}, we may
	immediately conclude the following corollary.
	
	\begin{cor}
		The polyhedral tree complex is contractible and admits a free and 
		cocompact action by the pure mapping class group for the 
		$n$-punctured plane. The quotient by this action, a finite cell
		complex with $\planar_n$ as its face poset, is thus a 
		classifying space for $\PMod(\mathbb{C}_P)$. 
	\end{cor}

	\begin{figure}
		\centering
		\begin{tikzpicture}

\tikzstyle{bigdot}=[dot,inner sep=3pt]
\tikzstyle{disk}=[thick,shape=circle,draw,color=black,fill=yellow!30, inner sep = 1cm]

	\newcommand{\drawtree}{
	\node[disk] () at (0,0) {};
	\foreach \i in {1,2,3,4} {
		\node[bigdot, color = blue!70] (v\i) at (\i*90-45:1cm) {};
	}
	}
	
	\begin{scope}{color = black}
	\node[dot] (x5) at (0,0) {};
	\foreach \i in {1,2,3,4} {
		\node[dot] (x\i) at (\i*90-45:11cm) {};
	}
	\end{scope}
	
	\begin{scope}[GreenLine, line width = 1.5mm]
		\draw (x1) -- (x2) -- (x3) -- (x4) -- (x1);
		\draw (x1) -- (x5) -- (x3);
		\draw (x2) -- (x5) -- (x4);
	\end{scope}
	
	\begin{scope}[very thick]
	\begin{scope}[shift={(135:11cm)}]
		\drawtree
		\node[bigdot, fill = white] (v5) at (0,0) {};
		\draw (v2) -- (v3) -- (v5);
		\draw (v1) -- (v5) -- (v4);		
	\end{scope}
	
	\begin{scope}[shift={(45:11cm)}]
		\drawtree
		\node[bigdot, fill = white] (v5) at (0,0) {};
		\draw (v1) -- (v4) -- (v5);
		\draw (v2) -- (v5) -- (v3);		
	\end{scope}
	
	\begin{scope}[shift={(225:11cm)}]
		\drawtree
		\node[bigdot, fill = white] (v5) at (0,0) {};
		\draw (v4) -- (v1) -- (v5);
		\draw (v2) -- (v5) -- (v3);		
	\end{scope}
	
	\begin{scope}[shift={(-45:11cm)}]
		\drawtree
		\node[bigdot, fill = white] (v5) at (0,0) {};
		\draw (v3) -- (v2) -- (v5);
		\draw (v1) -- (v5) -- (v4);		
	\end{scope}
	
	\begin{scope}[shift={(0,0)}]
		\drawtree
		\node[bigdot, fill = white] (v5) at (0,0) {};
		\draw (v2) -- (v5) -- (v3);
		\draw (v1) -- (v5) -- (v4);		
	\end{scope}
	
	\begin{scope}[shift={(0,-10)}]
		\drawtree
		\draw (v4) -- (v1) -- (v2) -- (v3);
	\end{scope}
	
	\begin{scope}[shift={(0,10)}]
		\drawtree
		\draw (v1) -- (v4) -- (v3) -- (v2);
	\end{scope}
	
	\begin{scope}[shift={(-10,0)}]
		\drawtree
		\draw (v2) -- (v3) -- (v1) -- (v4);
	\end{scope}
	
	\begin{scope}[shift={(10,0)}]
		\drawtree
		\draw (v3) -- (v2) -- (v4) -- (v1);
	\end{scope}
	
	\begin{scope}[shift={(-2.4,5)}]
		\drawtree
		\draw (v3) -- (v2);
		\draw (v3) -- (v1);
		\draw (v3) -- (v4);
	\end{scope}
	
	\begin{scope}[shift={(-2.4,-5)}]
		\drawtree
		\draw (v1) -- (v2);
		\draw (v1) -- (v3);
		\draw (v1) -- (v4);
	\end{scope}
	
	\begin{scope}[shift={(2.4,5)}]
		\drawtree
		\draw (v4) -- (v2);
		\draw (v4) -- (v1);
		\draw (v4) -- (v3);
	\end{scope}
	
	\begin{scope}[shift={(2.4,-5)}]
		\drawtree
		\draw (v2) -- (v3);
		\draw (v2) -- (v1);
		\draw (v2) -- (v4);
	\end{scope}
	\end{scope}

\end{tikzpicture}
		\caption{The link of a vertex in $\mathcal{P}_4$}
		\label{fig:poly-vertex-link}
	\end{figure}
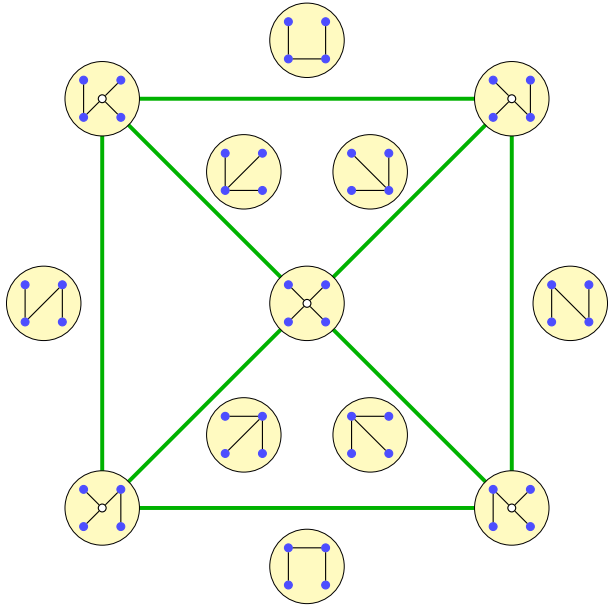

	\begin{example}
		To demonstrate the relative simplicity of the polyhedral cell
		structure, let $\Gamma$ be a planted $6$-tree such that $p(\Gamma)$
		is the planar $6$-tree shown in Figure~\ref{fig:n-tree}. Then
		$\Gamma$ labels a hexagonal prism in the polyhedral tree complex, and the
		barycentric subdivision of this prism (i.e. the corresponding subcomplex
		in the simplicial tree complex) consists of 72 tetrahedra which meet
		at a common vertex.
	\end{example}
	
	We close the section with two illustrative examples.
	
	\begin{example}
		The simplicial tree complex $\mathcal{S}_3$ is the bipartite tree
		$T_{2,3}$, whereas the polyhedral tree complex $\mathcal{P}_3$ is 
		the regular trivalent tree $T_3$ - see Figure~\ref{fig:simplicial-tree-complex}. 
		The quotient of $\mathcal{P}_3$ by the $\PMod(\mathbb{C}_P)$ action
		is the theta graph described in Remark~\ref{rem:planar-complex}.
	\end{example}
	
	\begin{example}
		The polyhedral tree complex $\mathcal{P}_4$ is a two-dimensional cell
		complex which can be thought of as the universal cover of the 
		complex with $\planar_4$ as its face poset, described in 
		Remark~\ref{rem:planar-complex}. The link of a vertex in $\mathcal{P}_4$
		is a graph with five vertices and eight edges - see Figure~\ref{fig:poly-vertex-link}.
	\end{example}
	
	\section{Marked and Unmarked Edges}
	
	The final section of this article concerns an equivalence relation on 
	the set of planted $n$-trees where elements are considered up to 
	contraction of unmarked edges.
		
	\begin{defn}
		Let $\Gamma_1 \leq \Gamma_2$ in $\planted_n$. If each component 
		in the subforest $F(\Gamma_1,\Gamma_2)$ is singly-marked, we 
		write $\Gamma_1 \leq_m \Gamma_2$. If 
		each component in the subforest is unmarked, then
		we write $\Gamma_1 \leq_u \Gamma_2$.
	\end{defn}

	\begin{lem}\label{lem:unmarked-factor}
		Suppose that $\Gamma_1 \leq \Gamma_2$ in $\planted_n$. Then there
		is a unique $\Gamma_3 \in \planted_n$ such that $\Gamma_1 \leq_u \Gamma_3$
		and $\Gamma_3 \leq_m \Gamma_2$.
	\end{lem}

	\begin{proof}
		If some subforest of $\Gamma_1$ can be contracted to obtain $\Gamma_2$, 
		then we may instead contract only the unmarked edges to obtain $\Gamma_3 \in \planted_n$ 
		such that $\Gamma_1 \leq_u \Gamma_3$ and $\Gamma_3 \leq_m \Gamma_2$.
		Moreover, $\Gamma_3$ is the unique tree with this property; if we 
		contracted a proper subset of the unmarked edges to
		obtain $\Gamma_3$, then we would need to contract such an edge in
		$\Gamma_3$ to obtain $\Gamma_2$, meaning that $\Gamma_3 \leq \Gamma_2$,
		but $\Gamma_3 \not{\leq_m} \Gamma_2$.
	\end{proof}

	\begin{lem}\label{lem:unmarked-meet}
		Suppose that $\Gamma_u \leq_u \Gamma$ and $\Gamma_m \leq_m \Gamma$ in
		$\planted_n$. Then the meet $\Gamma_u \wedge \Gamma_m$ exists in 
		$\planted_n$, and $\Gamma_u\wedge\Gamma_m \leq_m \Gamma_u$
		and $\Gamma_u\wedge\Gamma_m \leq_u \Gamma_m$.
	\end{lem}

	\begin{proof}
		If $\Gamma_u$ is obtained from $\Gamma$ by expanding
		some collection of unmarked vertices into subtrees and
		$\Gamma_m$ is obtained from $\Gamma$ by expanding
		some collection of marked vertices into subtrees, each
		containing exactly one marked point, then these
		expansions necessarily occur at disjoint sets of vertices, 
		so they do not interfere with one another and may be performed 
		simultaneously to obtain $\Gamma_u \wedge \Gamma_m$.
		By construction, we can further see that 
		$\Gamma_u\wedge\Gamma_m \leq_m \Gamma_u$
		and $\Gamma_u\wedge\Gamma_m \leq_u \Gamma_m$.
	\end{proof}

	These lemmas allow us to examine a modified version of the planted tree
	poset, in which we consider $n$-trees up to contraction of unmarked edges.
	
	\begin{defn}
		A planted $n$-tree is \emph{reduced} if it has no edges
		between unmarked vertices. For each planted $n$-tree $\Gamma$, there
		is a unique reduced tree in $\upper(\Gamma)$, denoted $r(\Gamma)$,
		obtained by contracting all edges between unmarked vertices.
		Define an equivalence relation on $\planted_n$ by
		declaring $\Gamma_1 \sim \Gamma_2$ if and only if 
		$r(\Gamma_1) = r(\Gamma_2)$, let $[\Gamma]$ denote the 
		equivalence class of $\Gamma$ under this relation (for which $r(\Gamma)$ is
		the maximum element) and let $\reduced_n$ denote the set of equivalence
		classes $\planted_n / \sim$.
	\end{defn}

	\begin{lem}\label{lem:reduced-poset}
		Declaring that $[\Gamma_1] \leq [\Gamma_2]$ in $\reduced_n$ 
		if and only if $\Gamma_1^\prime \leq \Gamma_2^\prime$ in $\planted_n$ for
		some $\Gamma_1^\prime \in [\Gamma_1]$ and $\Gamma_2^\prime \in [\Gamma_2]$
		yields a partial order for $\reduced_n$.
	\end{lem}

	\begin{proof}
		By Lemma~\ref{lem:unmarked-factor}, we know that
		if $\Gamma_1^\prime \in [\Gamma_1]$ and $\Gamma_2^\prime \in [\Gamma_2]$
		with $\Gamma_1^\prime \leq \Gamma_2^\prime$, then there is some 
		$\Gamma_1^{\prime\prime} \in [\Gamma_1]$ with 
		$\Gamma_1^{\prime\prime} \leq_m \Gamma_2^\prime$. Thus, it suffices
		to consider contractions involving marked edges in $\planted_n$.
		
		It follows immediately from the definition that this relation is
		reflexive, and anti-symmetry follows from our remark above:
		if $[\Gamma_1] \leq [\Gamma_2]$, then 
		$\Gamma_1^\prime \leq_m \Gamma_2^\prime$ for some 
		$\Gamma_1^\prime,\Gamma_2^\prime \in \planted_n$, 
		and so
		$\Gamma_1^\prime$ has at least as many edges with exactly one marked endpoint
		as $\Gamma_2^\prime$ does. Since the number of such edges is constant
		in each equivalence class, we know that $[\Gamma_2] \leq [\Gamma_1]$ only
		if $\Gamma_1^\prime$ and $\Gamma_2^\prime$ actually have the same number
		of edges with one marked endpoint, in which case $[\Gamma_1] = [\Gamma_2]$.
		
		All that remains is to demonstrate transitivity. Suppose that 
		$[\Gamma_1] \leq [\Gamma_2]$ and 
		$[\Gamma_2] \leq [\Gamma_3]$ in $\reduced_n$. Without loss of generality,
		we may write $\Gamma_1 \leq_m \Gamma_2$ and $\Gamma_2^\prime \leq_m \Gamma_3$,
		where $\Gamma_2^\prime \in [\Gamma_2]$. By Lemma~\ref{lem:unmarked-meet},
		we know that $\Gamma_1 \wedge \Gamma_2^\prime \leq_u \Gamma_1$ and
		$\Gamma_1 \wedge \Gamma_2^\prime \leq_m \Gamma_2^\prime$. By transitivity
		of the partial order for $\planted_n$, we then have that
		$\Gamma_1 \wedge \Gamma_2^\prime \leq_m \Gamma_3$, so $[\Gamma_1] \leq [\Gamma_3]$
		and we are done.
	\end{proof}

	\begin{defn}\label{def:reduced-poset}
		Equipped with the partial order from Lemma~\ref{lem:reduced-poset},
		we refer to $\reduced_n$ as the \emph{reduced tree poset}.
	\end{defn}

	\begin{rem}\label{rem:slide-and-split}
		Suppose that $\Gamma_1$ and $\Gamma_2$ are reduced $n$-trees and that
		$[\Gamma_1] < [\Gamma_2]$ in $\reduced_n$. There are three possibilities 
		for the relationship between $\Gamma_1$ and $\Gamma_2$.
		\begin{enumerate}
			\item (contraction) If $\Gamma_1 < \Gamma_2$ in $\planted_n$, then 
			$\Gamma_2$ is obtained from $\Gamma_1$ by contracting one 
			singly-marked edge. 
		\end{enumerate}
		If $\Gamma_2$ is not a contraction of $\Gamma_1$, then there exist
		$\Gamma_1^\prime$ and $\Gamma_2^\prime$ with 
		$\Gamma_1^\prime \leq_u \Gamma_1$ and $\Gamma_2^\prime \leq_u \Gamma_2$
		such that $\Gamma_2^\prime$ can be obtained from $\Gamma_1^\prime$
		by contracting some number of singly-marked edges. Since each can be contracted
		independently, it suffices to consider the case when there is only one
		contracted
		singly-marked edge $e$ in $\Gamma_1^\prime$.
		\begin{enumerate}[resume]
			\item (slide) If the unmarked end of $e$ is adjacent to 
			exactly one unmarked vertex,
			then $\Gamma_2$ is obtained from $\Gamma_1$ by ``sliding''
			some number of edges along $e$.
			\item (split) If the unmarked end of $e$ is adjacent to multiple 
			unmarked vertices, then $\Gamma_2$ is obtained from 
			$\Gamma_1$ by ``splitting''
			$e$ into some number of copies with different unmarked endpoints.
		\end{enumerate}
		Thus, when $[\Gamma_1] < [\Gamma_2]$ in $\reduced_n$, we know that
		$\Gamma_2$ is obtained from $\Gamma_1$ by a sequence of contractions, slides, 
		and splits along singly-marked edges. See Figure~\ref{fig:slide-and-split}.
	\end{rem}

	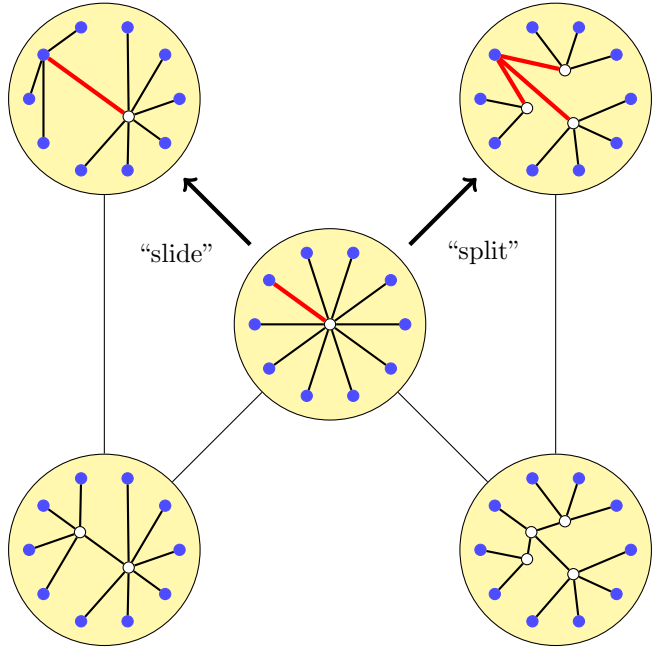
\begin{figure}
		\centering
		\begin{tikzpicture}

	
	\node [shape=circle,draw,fill=yellow!40,inner sep = 0.9cm] (pa1) at (-3,3) {};
	\node [shape=circle,draw,fill=yellow!40,inner sep = 0.9cm] (pa2) at (3,3) {};
	
	\node [shape=circle,draw,fill=yellow!40,inner sep = 0.9cm] (pb) at (0,0) {};
	
	\node [shape=circle,draw,fill=yellow!40,inner sep = 0.9cm] (pc1) at (-3,-3) {};
	\node [shape=circle,draw,fill=yellow!40,inner sep = 0.9cm] (pc2) at (3,-3) {};
	
	
	\begin{scope}[shift={(-3,3)}]
		\foreach \i in {1,...,10} {
			\node [dot,color=blue!70] (v\i)	at (36*\i:1cm) {};
		}
		\node [dot,fill=white] (u2) at (36*9:0.4cm) {};
		
		\begin{scope}[thick]
			\draw (v4) -- (v3);
			\draw (v4) -- (v5);
			\draw (v4) -- (v6);
			
			\draw[RedLine, ultra thick] (v4) -- (u2);
			
			\draw (u2) -- (v1);
			\draw (u2) -- (v2);
			\draw (u2) -- (v7);
			\draw (u2) -- (v8);
			\draw (u2) -- (v9);
			\draw (u2) -- (v10);
		\end{scope}
	\end{scope}
	
	\begin{scope}[shift={(3,3)}]
	\foreach \i in {1,...,10} {
		\node [dot,color=blue!70] (v\i)	at (36*\i:1cm) {};
	}
	\node [dot,fill=white] (u2) at (36*5.5:0.4cm) {};
	\node [dot,fill=white] (u3) at (36*8.5:0.4cm) {};
	\node [dot,fill=white] (u4) at (36*2:0.4cm) {};
	
	\begin{scope}[thick]
		\draw[RedLine, ultra thick] (v4) -- (u2);
		\draw[RedLine, ultra thick] (v4) -- (u3);
		\draw[RedLine, ultra thick] (v4) -- (u4);
		\draw (u2) -- (v5);
		\draw (u2) -- (v6);
		\draw (u3) -- (v7);
		\draw (u3) -- (v8);
		\draw (u3) -- (v9);
		\draw (u3) -- (v10);
		\draw (u4) -- (v1);
		\draw (u4) -- (v2);
		\draw (u4) -- (v3);
	\end{scope}
	\end{scope}
	
	
	\begin{scope}[shift={(0,0)}]
	\foreach \i in {1,...,10} {
		\node [dot,color=blue!70] (v\i)	at (36*\i:1cm) {};
	}
	\node [dot,fill=white] (u1) at (0,0) {};
	
	\begin{scope}[thick]
		\draw (u1) -- (v1);
		\draw (u1) -- (v2);
		\draw (u1) -- (v3);
		\draw[RedLine, ultra thick] (u1) -- (v4);
		\draw (u1) -- (v5);
		\draw (u1) -- (v6);
		\draw (u1) -- (v7);
		\draw (u1) -- (v8);
		\draw (u1) -- (v9);
		\draw (u1) -- (v10);
	\end{scope}
	\end{scope}
	
	
	\begin{scope}[shift={(-3,-3)}]
	\foreach \i in {1,...,10} {
		\node [dot,color=blue!70] (v\i)	at (36*\i:1cm) {};
	}
	\node [dot,fill=white] (u1) at (36*4:0.4cm) {};
	\node [dot,fill=white] (u2) at (36*9:0.4cm) {};
	
	\begin{scope}[thick]
	\draw (u1) -- (v3);
	\draw (u1) -- (v4);
	\draw (u1) -- (v5);
	\draw (u1) -- (v6);
	
	\draw (u2) -- (v1);
	\draw (u2) -- (v2);
	\draw (u2) -- (v7);
	\draw (u2) -- (v8);
	\draw (u2) -- (v9);
	\draw (u2) -- (v10);
	
	\draw (u1) -- (u2);
	\end{scope}
	\end{scope}
	
	\begin{scope}[shift={(3,-3)}]
	\foreach \i in {1,...,10} {
		\node [dot,color=blue!70] (v\i)	at (36*\i:1cm) {};
	}
	\node [dot,fill=white] (u1) at (36*4:0.4cm) {};
	\node [dot,fill=white] (u2) at (36*5.5:0.4cm) {};
	\node [dot,fill=white] (u3) at (36*8.5:0.4cm) {};
	\node [dot,fill=white] (u4) at (36*2:0.4cm) {};
	
	\begin{scope}[thick]
	\draw (u1) -- (v4);
	\draw (u1) -- (u2);
	\draw (u1) -- (u3);
	\draw (u1) -- (u4);
	\draw (u2) -- (v5);
	\draw (u2) -- (v6);
	\draw (u3) -- (v7);
	\draw (u3) -- (v8);
	\draw (u3) -- (v9);
	\draw (u3) -- (v10);
	\draw (u4) -- (v1);
	\draw (u4) -- (v2);
	\draw (u4) -- (v3);
	\end{scope}
	\end{scope}
	
	
	\begin{scope}[line join = bevel]
		\draw (pc1) -- (pb);
		\draw (pc2) -- (pb);
		
		\draw (pa1) -- (pc1);
		\draw (pa2) -- (pc2);
	\end{scope}
	
	\draw[ultra thick, ->] (45:1.5cm) -- (45:2.75cm);
	\draw[ultra thick, ->] (135:1.5cm) -- (135:2.75cm);
	
	\node () at (25:2.25cm) {``split''};
	\node () at (155:2.25cm) {``slide''};

\end{tikzpicture}
		\caption{A slide and a split for $\ustar_{10}$. The thin lines connecting
		trees indicate the partial order in $\planted_{10}$, and the relevant
		singly-marked edge is thicker and colored red.}
		\label{fig:slide-and-split}
	\end{figure}

	\begin{rem}\label{rem:reduced-disk-rep}
		Let $\Gamma$ be a reduced $n$-tree and suppose that $[\Gamma]$ lies
		in either the lower set or upper set for $[\mstar_n]$ or $[\ustar_n]$.
		By Remark~\ref{rem:slide-and-split}, we then know that $\Gamma$
		is related to either a marked or unmarked star by a sequence of 
		contractions, slides, and splits. In the same spirit as
		Remark~\ref{rem:star-drawing}, note that each of these three
		operations (and their reverses) do not disturb the canonical
		planar embedding of a star onto a disk. Thus, 
		$\Gamma$ can be represented in the plane as a tree with certain
		marked points on the unit circle ($v_1,\ldots,v_{n-1}$ if $\Gamma$
		is related to $\mstar_n$ and $v_1,\ldots,v_n$ if related to $\ustar_n$)
	 	and the rest of the tree on the interior of the unit disk.
	\end{rem}

	\begin{figure}
		\centering
		\begin{tikzpicture}

	\draw (-3,0) -- (0,3) -- (3,0);
	
	\draw[line width = 2mm, white] (-3,0) -- (-3,3) -- (0,0) -- (3,3)
	-- (3,0);
	\draw (-3,0) -- (-3,3) -- (0,0) -- (3,3)
	-- (3,0);
	\draw (0,6) -- (-3,3);
	\draw (0,6) -- (0,3);
	\draw (0,6) -- (3,3);
	
	\begin{scope}[shift={(0,6)}]
		\filldraw[color=gray!50, fill=green!20] (0,0) circle (1cm);
	
		\foreach \i in {1,2,3} {
			\node [dot,color=blue!70] (v\i)	at (120*\i-30:1cm) {};
		}
		\node [dot,color=blue!70] (v4) at (0,0) {};
		
		\begin{scope}[thick]
			\draw (v4) -- (v1);
			\draw (v4) -- (v2);
			\draw (v4) -- (v3);
		\end{scope}
	\end{scope}
	
	\foreach \j in {1,2,3} {
		\begin{scope}[shift={(6-3*\j,0)}]
			\filldraw[color=gray!50, fill=yellow!40] (0,0) circle (1cm);
			\filldraw[color=green!20,fill=green!20] (120*\j+210:0.5cm) -- (120*\j-30:1cm) 
			-- (120*\j+90:1cm) -- cycle;
			\filldraw[color=green!20,fill=green!20] (0,0) -- (120*\j-30:1cm) 
			arc [start angle = 120*\j-30, delta angle = 120, radius = 1cm] -- cycle;
			\draw[color=gray!50] (0,0) circle (1cm);
			
			\foreach \i in {1,2,3} {
				\node [dot,color=blue!70] (v\i)	at (120*\i-30:1cm) {};
			}
			\node [dot,fill=white] (u1) at (120*\j+210:0.5cm) {};
			\node [dot,color=blue!70] (v4) at (0,0) {};
		
			\begin{scope}[thick]
				\draw (v1) -- (u1);
				\draw (v2) -- (u1);
				\draw (v3) -- (u1);
				\draw (v4) -- (u1);
			\end{scope}
		\end{scope}
		
		\begin{scope}[shift={(3*\j-6,3)}]
			\filldraw[color=gray!50, fill=yellow!40] (0,0) circle (1cm);
			\filldraw[color=green!20,fill=green!20] (0,0) -- (120*\j+90:1cm) 
			arc [start angle = 120*\j+90, delta angle = 240, radius = 1cm] -- cycle;
			\filldraw[color=green!20,fill=green!20] (0,0) -- (120*\j+90:1cm)
			-- (120*\j-30:1cm) -- cycle;			
			\draw[color=gray!50] (0,0) circle (1cm);
			
			\foreach \i in {1,2,3} {
				\node [dot,color=blue!70] (v\i)	at (120*\i-150:1cm) {};
			}	
			\node [dot,fill=white] (v4) at (120*\j+30:0.5cm) {};
		
			\begin{scope}[thick]
				\draw (v4) -- (v1);
				\draw (v4) -- (v2);
				\draw (v4) -- (v3);
			\end{scope}
		
			\node [dot,color=blue!70] (u1) at (0,0) {};
		
		\end{scope}
	}

\end{tikzpicture}
		\caption{The lower set $\low([\mstar_4])$, where each equivalence
			class is represented by its maximal element and the regions
			adjacent to the unique interior marked point are shaded
			green}
		\label{fig:reduced-lower-set}			
	\end{figure}
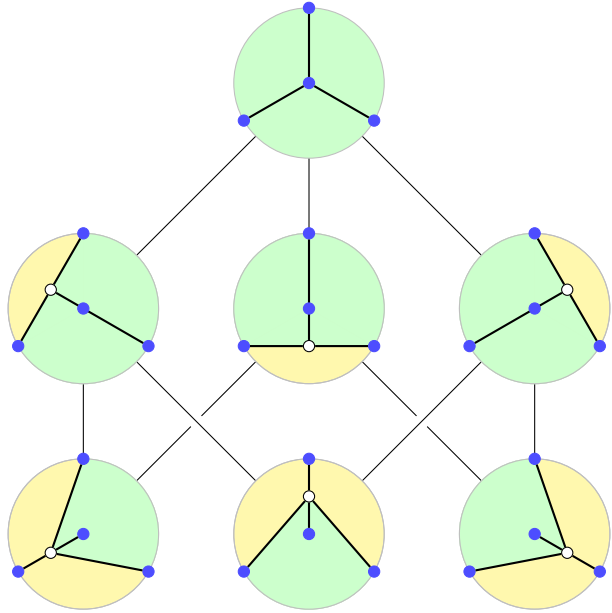

	For the remainder of this article, we consider the structure of lower sets
	and upper sets in the reduced tree poset.
	
	\begin{lem}\label{lem:lower-mstar-reduced}
		The lower set $\low([\mstar_n])$ is isomorphic to $\bool_{n-1}^\ast$.
	\end{lem}
	
	\begin{proof}
		Let $[\Gamma] \in \low([\mstar_n])$, where $\Gamma$ is chosen to be
		the maximum element of its equivalence class.
		By Remark~\ref{rem:reduced-disk-rep}, we may draw $\Gamma$ so that 
		the marked points $v_1,\ldots,v_{n-1}$ are on the boundary of the unit
		disk and $v_n$ lies at the origin. Then $\Gamma$ divides the disk into
		$n-1$ regions; let $\Omega_i$ denote the region which is incident to
		$v_i$ and $v_{i+1}$. The vertex $v_n$ is then incident to some nonempty 
		subset of $\{\Omega_1,\ldots,\Omega_{n-1}\}$, which we can naturally identify
		with an element of $\bool_{n-1}^\ast$. Note that this subset is actually
		well-defined on the equivalence class $[\Gamma]$ since the contraction or 
		expansion of unmarked edges does not change which regions are incident to $v_n$. 
		Define $f\colon \low([\mstar_n]) \to \bool_{n-1}^\ast$ to be the function
		which sends $[\Gamma]$ to this corresponding subset of 
		$\{1,\ldots,n-1\}$.
		If $\Gamma_1$ and $\Gamma_2$ are reduced $n$-trees such that 
		$[\Gamma_1],[\Gamma_2] \in \low([\mstar_n])$ and $[\Gamma_1] \leq [\Gamma_2]$,
		then $\Gamma_2$ is obtained from $\Gamma_1$ by a sequence of contractions, 
		slides, and/or splits by Remark~\ref{rem:slide-and-split}, and since
		each of these increases the number of regions incident to $v_n$, it follows
		that $f(\Gamma_1) \subseteq f(\Gamma_2)$ and thus $f$ is order-preserving.
		
		Next, we build an inverse for $f$ by constructing a reduced $n$-tree for each 
		nonempty $A \subset \{1,\ldots,n-1\}$. Begin with $v_1,\ldots,v_{n-1}$ labeling
		points around the boundary of the unit disk and let $v_n$ be at the origin.
		For each maximal string of consecutive cyclically-ordered elements in 
		$\{1,\ldots,n-1\} - A$, introduce a new unmarked vertex near the 
		corresponding boundary vertices, then connect it to those vertices and
		$v_n$. For example, if $i$ and $i+k$ (reduced mod $n-1$) are in $A$,
		but $i+1,i+2,\ldots,i+k-1$ are not, then we would introduce a single
		unmarked vertex associated to this string of elements and connect it
		to $v_{i+1},v_{i+2},\ldots,v_{i+k-1}$ and $v_n$. Finally, we connect
		any remaining boundary vertices directly to $v_n$. The result is 
		a reduced planted $n$-tree such that $v_n$ is adjacent to the 
		desired regions. Moreover, this inverse is order-preserving as well - see 
		Figure~\ref{fig:order-preserving-inverse} for an illustrative example.
		Therefore, $f$ is an isomorphism from $\low(\mstar_n)$ to $\bool_{n-1}^\ast$
		and the proof is complete.
	\end{proof}

	\begin{figure}
		\centering
		\begin{tikzpicture}
	
	\begin{scope}[shift={(-2,0)}]
		\filldraw[color=gray!50, fill=yellow!40] (0,0) circle (1cm);
		
		\filldraw[color=gray!50, fill=green!20] 
		(0,0) -- (45:1cm) arc [start angle=45, end angle=90, radius=1cm] 
		-- (112.5:0.5cm) -- (0,0);
		
		\filldraw[color=gray!50, fill=green!20] 
		(0,0) -- (45*2.5:0.5cm) -- (135:1cm) arc [start angle=135, end angle=180, radius=1cm]
		-- (45*4.5:0.5cm) -- (0,0);
		
		\filldraw[color=gray!50, fill=green!20] 
		(0,0) -- (45*4.5:0.5cm) -- (225:1cm) arc [start angle=225, end angle=270, radius=1cm]
		-- (45*7:0.5cm) -- (0,0);
		
		\filldraw[color=gray!50, fill=green!20] 
		(0,0) -- (45*7:0.5cm) -- (0:1cm) arc [start angle=0, end angle=45, radius=1cm]
		-- (0,0);

		\foreach \i in {1,...,8} {
			\node [dot,color=blue!70] (v\i)	at (45*\i:1cm) {};
			\node () at (45*\i:1.25cm) {\tiny $\i$};
		}
		\node [dot,color=blue!70] (v9) at (0,0) {};
		\node () at (157.5:0.25cm) {\tiny $9$};
		\node [dot,fill=white] (u1) at (45*2.5:0.5cm) {};
		\node [dot,fill=white] (u2) at (45*4.5:0.5cm) {};
		\node [dot,fill=white] (u3) at (45*7:0.5cm) {};
		
		\begin{scope}[thick]
			\draw (v9) -- (u1);
			\draw (v9) -- (u2);
			\draw (v9) -- (u3);
			\draw (v9) -- (v1);
	
			\draw (u1) -- (v2);
			\draw (u1) -- (v3);
			
			\draw (u2) -- (v4);
			\draw (u2) -- (v5);
			
			\draw (u3) -- (v6);
			\draw (u3) -- (v7);
			\draw (u3) -- (v8);
		\end{scope}
	\end{scope}
	
	\begin{scope}[shift={(2,0)}]
		\filldraw[color=gray!50, fill=yellow!40] (0,0) circle (1cm);
		
		\filldraw[color=gray!50, fill=green!20] 
		(0,0) -- (45:1cm) arc [start angle=45, end angle=90, radius=1cm] 
		-- (112.5:0.5cm) -- (0,0);
		
		\filldraw[color=gray!50, fill=green!20] 
		(0,0) -- (45*2.5:0.5cm) -- (135:1cm) arc [start angle=135, end angle=180, radius=1cm]
		-- (45*4.5:0.5cm) -- (0,0);
		
		\filldraw[color=gray!50, fill=green!20] 
		(0,0) -- (45*4.5:0.5cm) -- (225:1cm) arc [start angle=225, end angle=270, radius=1cm]
		-- (0,0);
		
		\filldraw[color=gray!50, fill=green!20] 
		(0,0) -- (270:1cm) arc [start angle=270, end angle=315, radius=1cm]
		-- (45*7.5:0.5cm) -- (0,0);
		
		\filldraw[color=gray!50, fill=green!20] 
		(0,0) -- (45*7.5:0.5cm) -- (0:1cm) arc [start angle=0, end angle=45, radius=1cm]
		-- (0,0);
		
	\foreach \i in {1,...,8} {
		\node [dot,color=blue!70] (v\i)	at (45*\i:1cm) {};
		\node () at (45*\i:1.25cm) {\tiny $\i$};
	}
	\node [dot,color=blue!70] (v9) at (0,0) {};
	\node () at (157.5:0.25cm) {\tiny $9$};
	\node [dot,fill=white] (u1) at (45*2.5:0.5cm) {};
	\node [dot,fill=white] (u2) at (45*4.5:0.5cm) {};
	\node [dot,fill=white] (u3) at (45*7.5:0.5cm) {};
	
	\begin{scope}[thick]
		\draw (v9) -- (u1);
		\draw (v9) -- (u2);
		\draw (v9) -- (u3);
		\draw (v9) -- (v1);
		\draw (v9) -- (v6);
		
		\draw (u1) -- (v2);
		\draw (u1) -- (v3);
		
		\draw (u2) -- (v4);
		\draw (u2) -- (v5);
		
		\draw (u3) -- (v7);
		\draw (u3) -- (v8);
	\end{scope}
	\end{scope}

\end{tikzpicture}
		\caption{These reduced $9$-trees correspond to the subsets
			$\{1,3,5,8\}\subseteq\{1,3,5,6,8\}$ in 
			$\bool_{8}^\ast$, where the relevant regions are 
			shaded green. Note that the second
			tree can be obtained from the first by performing a slide.}
		\label{fig:order-preserving-inverse}
	\end{figure}
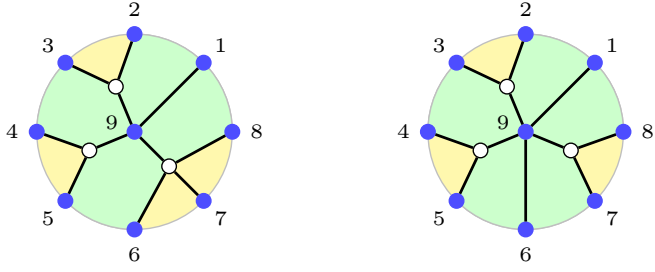

	Using Lemmas~\ref{lem:star-cyclo} and \ref{lem:lower-mstar-reduced},
	together with the fact that $\bool_{n-1}^\ast$ is the face poset for
	an $(n-2)$-simplex,
	we can see that the combinatorial process of passing from 
	planted $n$-trees to reduced $n$-trees induces a topological
	transformation of a cyclohedron into a simplex of equal dimension.
	This procedure has been described previously by Devadoss:
	the $n$-dimensional cyclohedron can be obtained from the Coxeter simplex
	for the affine Coxeter group of type $\widetilde{A}_{n-1}$ by iteratively 
	truncating a certain sequence of faces \cite{devadoss03}. 
	
	More generally, Theorem~\ref{thm:poset-product} and
	Lemma~\ref{lem:planted-planar-iso} imply that 
	$\low(\Gamma)$ is the face poset for a product of associahedra and
	cyclohedra for any planted $n$-tree $\Gamma$. The following theorem
	gives the corresponding topological transformation for
	$\low([\Gamma])$.

	\begin{thm}[Theorem~\ref{thma:reduced}, first claim]
		\label{thm:lower-reduced}
		Let $\Gamma$ be a reduced $n$-tree with marked interior vertices
		$v_1,\ldots,v_\ell$. Then
		\[
			\low([\Gamma]) \cong \prod_{i=1}^\ell \bool_{\val(v_i)}^\ast
		\]
		and so the lower set $\low([\Gamma])$ is 
		isomorphic to the face poset for a product of simplices.
	\end{thm}
	
	\begin{proof}
		The main idea for this proof is similar to that of 
		Theorem~\ref{thm:poset-product}: the structure of $\low([\Gamma])$
		is determined by what the neighborhoods of interior vertices
		for $\Gamma$ look like. The additional wrinkle in this case 
		is that we are now dealing with equivalence classes of trees,
		so it is not immediately clear that such an argument is valid.
		
		To that end, note that $[\ustar_n]$ is the unique element of 
		$\low([\ustar_n])$ since the only trees below $\ustar_n$ in 
		the partial order are obtained by expanding the unique 
		unmarked vertex, and hence lie in the same equivalence class.
		So when considering elements of $\low([\Gamma])$, we need
		only concern ourselves with marked interior vertices. Note
		also that the neighborhood of such an interior vertex remains
		unchanged under contraction of unmarked edges, so there is a
		well-defined meaning for the neighborhoods of marked interior
		vertices of $[\Gamma]$.
		
		From here, the proof is identical to that of 
		Theorem~\ref{thm:poset-product}. Since the neighborhoods of
		interior vertices are well-defined up to equivalence and 
		the expansions of distinct interior vertices can be considered
		independently, we know that $\low([\Gamma])$ is isomorphic
		to the product of the lower sets of the neighborhoods of 
		its interior vertices. Finally, we know that
		$\low([\ustar_n])$ is trivial by our argument above and that
		$\low([\mstar_n]) \cong \bool_{n-1}^\ast$ by 
		Lemma~\ref{lem:lower-mstar-reduced}, so $\low([\Gamma])$ is
		isomorphic to the face poset of a product of simplices.
	\end{proof}

	The second half of Theorem~\ref{thma:reduced} concerns the upper 
	sets of $\reduced_n$, which can also be connected to a 
	well-understood object: the poset of noncrossing hypertrees.
	We give a specialized definition here; more background
	may be found in \cite{hypertrees}.

	\begin{defn}
		Let $n\geq 3$, define $z_k = e^{i\pi k / n}$ for each integer $k$, 
		and let $V = \{z_1,\ldots,z_n\}$.
		A \emph{hyperedge} is a subset of $V$ with at least two elements
		and can be drawn in the plane as the convex hull of its vertices;
		a \emph{hypergraph} is a collection of hyperedges on the vertex set $V$.
		A \emph{noncrossing hypertree} is a hypergraph with no embedded loops 
		such that the intersection of any pair of hyperedges is either empty
		or consists of a single vertex. The set of all
		noncrossing hypertrees on this vertex set admits a partial order: 
		if $\sigma$ and $\tau$ are noncrossing hypertrees, then we say that
		$\sigma \leq \tau$ if each hyperedge of $\sigma$ is a subset of
		a hyperedge of $\tau$. Denote the partially ordered set of 
		noncrossing hypertrees by $\ncht_n$.
	\end{defn}

	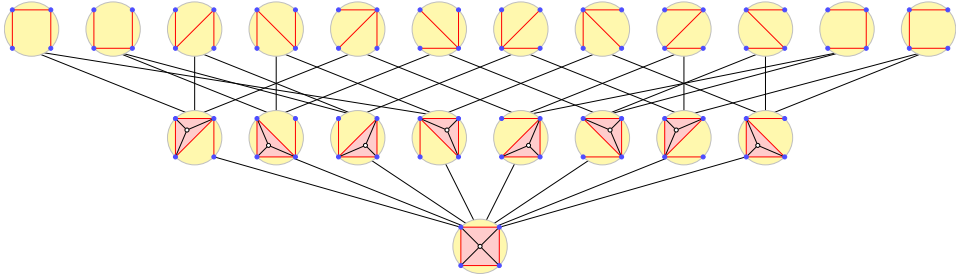
\begin{figure}
		\centering
		\begin{tikzpicture}

	\draw (0,0.5) -- (-10.5,3.5);
	\draw (0,0.5) -- (-7.5,3.5);
	\draw (0,0.5) -- (-4.5,3.5);
	\draw (0,0.5) -- (-1.5,3.5);
	\draw (0,0.5) -- (1.5,3.5);
	\draw (0,0.5) -- (4.5,3.5);
	\draw (0,0.5) -- (7.5,3.5);
	\draw (0,0.5) -- (10.5,3.5);
	
	\draw (-10.5,4.8) -- (-16.5,7.2);
	\draw (-10.5,4.8) -- (-10.5,7.2);
	\draw (-10.5,4.8) -- (-4.5,7.2);
	
	\draw (-7.5,4.8) -- (-13.5,7.2);
	\draw (-7.5,4.8) -- (-7.5,7.2);
	\draw (-7.5,4.8) -- (-1.5,7.2);
	
	\draw (-4.5,4.8) -- (-13.5,7.2);
	\draw (-4.5,4.8) -- (-10.5,7.2);
	\draw (-4.5,4.8) -- (1.5,7.2);
	
	\draw (-1.5,4.8) -- (-16.5,7.2);
	\draw (-1.5,4.8) -- (-7.5,7.2);
	\draw (-1.5,4.8) -- (4.5,7.2);
	
	\draw (1.5,4.8) -- (-4.5,7.2);
	\draw (1.5,4.8) -- (7.5,7.2);
	\draw (1.5,4.8) -- (13.5,7.2);
	
	\draw (4.5,4.8) -- (-1.5,7.2);
	\draw (4.5,4.8) -- (10.5,7.2);
	\draw (4.5,4.8) -- (13.5,7.2);
	
	\draw (7.5,4.8) -- (1.5,7.2);
	\draw (7.5,4.8) -- (7.5,7.2);
	\draw (7.5,4.8) -- (16.5,7.2);
	
	\draw (10.5,4.8) -- (16.5,7.2);
	\draw (10.5,4.8) -- (10.5,7.2);
	\draw (10.5,4.8) -- (4.5,7.2);
	
	\begin{scope}[shift={(-16.5,8)}]
		\filldraw[color=gray!50, fill=yellow!40] (0,0) circle (1cm);
		\foreach \i in {1,2,3,4} {
			\node [dot,color=blue!70] (v\i)	at (90*\i-45:1cm) {};
		}
		\begin{scope}[thick, RedLine]
			\draw (v4) -- (v1) -- (v2) -- (v3);
		\end{scope}
	\end{scope}
	
	\begin{scope}[shift={(-13.5,8)}]
		\filldraw[color=gray!50, fill=yellow!40] (0,0) circle (1cm);
		\foreach \i in {1,2,3,4} {
			\node [dot,color=blue!70] (v\i)	at (90*\i-45:1cm) {};
		}
		\begin{scope}[thick, RedLine]
			\draw (v1) -- (v4) -- (v3) -- (v2);
		\end{scope}
	\end{scope}
	
	\begin{scope}[shift={(-10.5,8)}]
		\filldraw[color=gray!50, fill=yellow!40] (0,0) circle (1cm);
		\foreach \i in {1,2,3,4} {
			\node [dot,color=blue!70] (v\i)	at (90*\i-45:1cm) {};
		}
		\begin{scope}[thick, RedLine]
			\draw (v2) -- (v3) -- (v1) -- (v4);
		\end{scope}
	\end{scope}
	
	\begin{scope}[shift={(-7.5,8)}]
		\filldraw[color=gray!50, fill=yellow!40] (0,0) circle (1cm);
		\foreach \i in {1,2,3,4} {
			\node [dot,color=blue!70] (v\i)	at (90*\i-45:1cm) {};
		}
		\begin{scope}[thick, RedLine]
			\draw (v1) -- (v4) -- (v2) -- (v3);
		\end{scope}
	\end{scope}
	
	\begin{scope}[shift={(-4.5,8)}]
		\filldraw[color=gray!50, fill=yellow!40] (0,0) circle (1cm);
		\foreach \i in {1,2,3,4} {
			\node [dot,color=blue!70] (v\i)	at (90*\i-45:1cm) {};
		}
		\begin{scope}[thick, RedLine]
			\draw (v1) -- (v2);
			\draw (v1) -- (v3);
			\draw (v1) -- (v4);
		\end{scope}
	\end{scope}
	
	\begin{scope}[shift={(-1.5,8)}]
		\filldraw[color=gray!50, fill=yellow!40] (0,0) circle (1cm);
		\foreach \i in {1,2,3,4} {
			\node [dot,color=blue!70] (v\i)	at (90*\i-45:1cm) {};
		}
		\begin{scope}[thick, RedLine]
			\draw (v4) -- (v1);
			\draw (v4) -- (v2);
			\draw (v4) -- (v3);
		\end{scope}
	\end{scope}
	
	\begin{scope}[shift={(1.5,8)}]
		\filldraw[color=gray!50, fill=yellow!40] (0,0) circle (1cm);
		\foreach \i in {1,2,3,4} {
			\node [dot,color=blue!70] (v\i)	at (90*\i-45:1cm) {};
		}
		\begin{scope}[thick, RedLine]
			\draw (v3) -- (v1);
			\draw (v3) -- (v2);
			\draw (v3) -- (v4);
		\end{scope}
	\end{scope}
	
	\begin{scope}[shift={(4.5,8)}]
		\filldraw[color=gray!50, fill=yellow!40] (0,0) circle (1cm);
		\foreach \i in {1,2,3,4} {
			\node [dot,color=blue!70] (v\i)	at (90*\i-45:1cm) {};
		}
		\begin{scope}[thick, RedLine]
			\draw (v2) -- (v1);
			\draw (v2) -- (v3);
			\draw (v2) -- (v4);
		\end{scope}
	\end{scope}
	
	\begin{scope}[shift={(7.5,8)}]
		\filldraw[color=gray!50, fill=yellow!40] (0,0) circle (1cm);
		\foreach \i in {1,2,3,4} {
			\node [dot,color=blue!70] (v\i)	at (90*\i-45:1cm) {};
		}
		\begin{scope}[thick, RedLine]
			\draw (v2) -- (v1) -- (v3) -- (v4);
		\end{scope}
	\end{scope}
	
	\begin{scope}[shift={(10.5,8)}]
		\filldraw[color=gray!50, fill=yellow!40] (0,0) circle (1cm);
		\foreach \i in {1,2,3,4} {
			\node [dot,color=blue!70] (v\i)	at (90*\i-45:1cm) {};
		}
		\begin{scope}[thick, RedLine]
			\draw (v1) -- (v2) -- (v4) -- (v3);
		\end{scope}
	\end{scope}
	
	\begin{scope}[shift={(13.5,8)}]
		\filldraw[color=gray!50, fill=yellow!40] (0,0) circle (1cm);
		\foreach \i in {1,2,3,4} {
			\node [dot,color=blue!70] (v\i)	at (90*\i-45:1cm) {};
		}
		\begin{scope}[thick, RedLine]
			\draw (v2) -- (v1) -- (v4) -- (v3);
		\end{scope}
	\end{scope}
	
	\begin{scope}[shift={(16.5,8)}]
		\filldraw[color=gray!50, fill=yellow!40] (0,0) circle (1cm);
		\foreach \i in {1,2,3,4} {
			\node [dot,color=blue!70] (v\i)	at (90*\i-45:1cm) {};
		}
		\begin{scope}[thick, RedLine]
			\draw (v1) -- (v2) -- (v3) -- (v4);
		\end{scope}
	\end{scope}
	
	
	\begin{scope}[shift={(-10.5,4)}]
		\filldraw[color=gray!50, fill=yellow!40] (0,0) circle (1cm);
		\foreach \i in {1,2,3,4} {
			\coordinate (v\i) at (90*\i-45:1cm);
		}
		\begin{scope}[line width=0.3mm,RedPoly]
			\filldraw (v1) -- (v2) -- (v3) -- cycle;
			\draw (v1) -- (v4);
		\end{scope}
		\node[dot,fill=white] (u) at (135:0.4cm) {};
		\begin{scope}[thick]
			\draw (v1) -- (u);
			\draw (v2) -- (u);
			\draw (v3) -- (u);
		\end{scope}
		\foreach \i in {1,2,3,4} {
			\node [dot,color=blue!70] () at (v\i) {};
		}
	\end{scope}
	
	\begin{scope}[shift={(-7.5,4)}]
		\filldraw[color=gray!50, fill=yellow!40] (0,0) circle (1cm);
		\foreach \i in {1,2,3,4} {
			\coordinate (v\i) at (90*\i-45:1cm);
		}
		\begin{scope}[line width=0.3mm,RedPoly]
			\filldraw (v2) -- (v3) -- (v4) -- cycle;
			\draw (v1) -- (v4);
		\end{scope}
		\node[dot,fill=white] (u) at (225:0.4cm) {};
		\begin{scope}[thick]
			\draw (v2) -- (u);
			\draw (v3) -- (u);
			\draw (v4) -- (u);
		\end{scope}
		\foreach \i in {1,2,3,4} {
			\node [dot,color=blue!70] () at (v\i) {};
		}
	\end{scope}
	
	\begin{scope}[shift={(-4.5,4)}]
		\filldraw[color=gray!50, fill=yellow!40] (0,0) circle (1cm);
		\foreach \i in {1,2,3,4} {
			\coordinate (v\i) at (90*\i-45:1cm);
		}
		\begin{scope}[line width=0.3mm,RedPoly]
			\filldraw (v1) -- (v3) -- (v4) -- cycle;
			\draw (v2) -- (v3);
		\end{scope}
		\node[dot,fill=white] (u) at (315:0.4cm) {};
		\begin{scope}[thick]
		\draw (v3) -- (u);
		\draw (v4) -- (u);
		\draw (v1) -- (u);
		\end{scope}
		\foreach \i in {1,2,3,4} {
			\node [dot,color=blue!70] () at (v\i) {};
		}
	\end{scope}
	
	\begin{scope}[shift={(-1.5,4)}]
		\filldraw[color=gray!50, fill=yellow!40] (0,0) circle (1cm);
		\foreach \i in {1,2,3,4} {
			\coordinate (v\i) at (90*\i-45:1cm);
		}
		\begin{scope}[line width=0.3mm,RedPoly]
			\filldraw (v1) -- (v2) -- (v4) -- cycle;
			\draw (v2) -- (v3);
		\end{scope}
		\node[dot,fill=white] (u) at (45:0.4cm) {};
		\begin{scope}[thick]
		\draw (v1) -- (u);
		\draw (v2) -- (u);
		\draw (v4) -- (u);
		\end{scope}
		\foreach \i in {1,2,3,4} {
			\node [dot,color=blue!70] () at (v\i) {};
		}
	\end{scope}
	
	\begin{scope}[shift={(1.5,4)}]
		\filldraw[color=gray!50, fill=yellow!40] (0,0) circle (1cm);
		\foreach \i in {1,2,3,4} {
			\coordinate (v\i) at (90*\i-45:1cm);
		}
		\begin{scope}[line width=0.3mm,RedPoly]
			\filldraw (v1) -- (v3) -- (v4) -- cycle;
			\draw (v1) -- (v2);
		\end{scope}
		\node[dot,fill=white] (u) at (315:0.4cm) {};
		\begin{scope}[thick]
		\draw (v3) -- (u);
		\draw (v4) -- (u);
		\draw (v1) -- (u);
		\end{scope}
		\foreach \i in {1,2,3,4} {
			\node [dot,color=blue!70] () at (v\i) {};
		}
	\end{scope}
	
	\begin{scope}[shift={(4.5,4)}]
		\filldraw[color=gray!50, fill=yellow!40] (0,0) circle (1cm);
		\foreach \i in {1,2,3,4} {
			\coordinate (v\i) at (90*\i-45:1cm);
		}
		\begin{scope}[line width=0.3mm,RedPoly]
			\filldraw (v1) -- (v2) -- (v4) -- cycle;
			\draw (v3) -- (v4);
		\end{scope}
		\node[dot,fill=white] (u) at (45:0.4cm) {};
		\begin{scope}[thick]
		\draw (v1) -- (u);
		\draw (v2) -- (u);
		\draw (v4) -- (u);
		\end{scope}
		\foreach \i in {1,2,3,4} {
			\node [dot,color=blue!70] () at (v\i) {};
		}
	\end{scope}
	
	\begin{scope}[shift={(7.5,4)}]
		\filldraw[color=gray!50, fill=yellow!40] (0,0) circle (1cm);
		\foreach \i in {1,2,3,4} {
			\coordinate (v\i) at (90*\i-45:1cm);
		}
		\begin{scope}[line width=0.3mm,RedPoly]
			\filldraw (v1) -- (v2) -- (v3) -- cycle;
			\draw (v3) -- (v4);
		\end{scope}
		\node[dot,fill=white] (u) at (135:0.4cm) {};
		\begin{scope}[thick]
		\draw (v1) -- (u);
		\draw (v2) -- (u);
		\draw (v3) -- (u);
		\end{scope}
		\foreach \i in {1,2,3,4} {
			\node [dot,color=blue!70] () at (v\i) {};
		}
	\end{scope}
	
	\begin{scope}[shift={(10.5,4)}]
		\filldraw[color=gray!50, fill=yellow!40] (0,0) circle (1cm);
		\foreach \i in {1,2,3,4} {
			\coordinate (v\i) at (90*\i-45:1cm);
		}
		\begin{scope}[line width=0.3mm,RedPoly]
			\filldraw (v2) -- (v3) -- (v4) -- cycle;
			\draw (v1) -- (v2);
		\end{scope}
		\node[dot,fill=white] (u) at (225:0.4cm) {};
		\begin{scope}[thick]
		\draw (v2) -- (u);
		\draw (v3) -- (u);
		\draw (v4) -- (u);
		\end{scope}
		\foreach \i in {1,2,3,4} {
			\node [dot,color=blue!70] () at (v\i) {};
		}
	\end{scope}
	
	
	\begin{scope}[shift={(0,0)}]
		\filldraw[color=gray!50, fill=yellow!40] (0,0) circle (1cm);
		\foreach \i in {1,2,3,4} {
			\coordinate (v\i) at (90*\i-45:1cm);
		}
		\begin{scope}[RedPoly]
			\filldraw[RedPoly] (v1) -- (v2) -- (v3) -- (v4) -- cycle;
		\end{scope}
		\node[dot,fill=white] (u) at (0,0) {};
		\begin{scope}[thick]
			\draw (u) -- (v1);
			\draw (u) -- (v2);
			\draw (u) -- (v3);
			\draw (u) -- (v4);
		\end{scope}
		\foreach \i in {1,2,3,4} {
			\node [dot,color=blue!70] () at (v\i) {};
		}
	\end{scope}
	
\end{tikzpicture}
		\caption{The upper set $\upper([\ustar_4])$, where each
			equivalence class is represented by its maximal element
			and the corresponding noncrossing hypertree is superimposed}
		\label{fig:reduced-upper-set}
	\end{figure}
	
	\begin{lem}\label{lem:upper-ustar-reduced}
		The upper set $\upper([\ustar_n])$ is isomorphic to the dual of $\ncht_n$.	
	\end{lem}
	
	\begin{proof}
		Let $[\Gamma] \in \upper([\ustar_n])$, where $\Gamma$ is chosen to be
		the maximum element of its equivalence class.
		By Remark~\ref{rem:reduced-disk-rep}, we may draw $\Gamma$ so that 
		the marked points $v_1,\ldots,v_n$ are on the boundary of the unit
		disk and the rest of the tree lies in the interior.
		Since $\Gamma$ is reduced, we know that each unmarked vertex $u$ with 
		valence $k$ in $\Gamma$ has a neighborhood $\nbhd(u)$ which is isomorphic 
		to $\ustar_k$, up to relabeling. The leaves of $\nbhd(u)$ then determine 
		a hyperedge, and the set of all such hyperedges gives a hypergraph on 
		the vertex set $V = \{v_1,\ldots,v_n\}$. In fact, this is a noncrossing 
		(since $\Gamma$ is embedded in the plane) hypertree (since $\Gamma$ 
		is a tree); let $f\colon \upper([\ustar_n]) \to \ncht_n^d$ be the function which 
		sends $[\Gamma]$ to the corresponding noncrossing hypertree, where
		$\ncht_n^d$ denotes the dual of the noncrossing hypertree poset.
		
		Constructing an inverse for $f$ is straightforward. Fix a noncrossing
		hypertree. For each hyperedge with at least three vertices, introduce an 
		unmarked vertex at the barycenter and connect it to each of the 
		(marked) vertices at the corners via a 
		straight edge. For each hyperedge with two vertices (i.e. a typical edge), 
		simply leave it as an edge between two marked vertices. 
		Since we began with a noncrossing hypertree, this graph is an embedded
		tree, and in particular it is reduced since no two unmarked vertices are
		adjacent. So $f$ is a bijection.
		
		Finally, we can see that $f$ is an order embedding. If $\Gamma_1$ and $\Gamma_2$
		are reduced trees and $[\Gamma_1] \leq [\Gamma_2]$ in $\upper([\ustar_n])$,
		then by Remark~\ref{rem:slide-and-split}, this is equivalent to saying that
		$\Gamma_2$ can be obtained from $\Gamma_1$ via a sequence of contractions, 
		slides, and splits. We can see by definition that each of these 
		serves to break apart the corresponding hyperedges into smaller pieces,
		and vice versa. Thus, $[\Gamma_1] \leq [\Gamma_2]$ in $\upper([\ustar_n])$ 
		if and only if $f([\Gamma_1]) \leq f([\Gamma_2])$ in $\ncht_n^d$, so 
		$f$ is a poset isomorphism.
	\end{proof}
	
	\begin{thm}[Theorem~\ref{thma:reduced}, second claim]
		\label{thm:upper-reduced}
		Let $\Gamma$ be a reduced $n$-tree with unmarked vertices
		$u_1,\ldots,u_k$. Then
		\[
			\upper([\Gamma]) \cong \prod_{i=1}^k \ncht_{\val(u_i)}.
		\]
	\end{thm}
	
	\begin{proof}
		The argument is similar to that of Theorem~\ref{thm:lower-reduced}.
		Note first that $\upper([\mstar_n])$ is trivial since no contractions,
		slides, or splits can be performed without unmarked vertices. 
		So the structure of $\upper([\Gamma])$ is determined entirely by the
		unmarked vertices of $\Gamma$, which are isolated since $\Gamma$ is assumed
		to be reduced. 
		
		For each unmarked vertex $u_i$ in $\Gamma$, the subtree $\nbhd(u_i)$ can be 
		identified with $\ustar_{\val(u_i)}$. As in the case of 
		Theorem~\ref{thm:lower-reduced}, these unmarked vertices contribute
		independently to the structure of $\upper([\Gamma])$. More specifically,
		we can see that 
		\[
			\upper([\Gamma]) \cong \prod_{i=1}^k \upper\left(\left[\ustar_{\val(u_i)}\right]\right)
			\cong \prod_{i=1}^k \ncht_{\val(u_i)},
		\]
		where the second isomorphism is due to Lemma~\ref{lem:upper-ustar-reduced}.
	\end{proof}

	\begin{rem}
		Each cell in the polyhedral tree complex is a product of 
		associahedra and cyclohedra by Theorem~\ref{thm:polyhedral-tree-complex},
		so Theorems~\ref{thm:lower-reduced} and \ref{thm:upper-reduced}
		suggest that by passing from planted $n$-trees to reduced $n$-trees,
		the polyhedral tree complex is transformed into a polysimplicial cell complex
		for which each vertex link is described by the poset of noncrossing hypertrees.
		This transformation can be viewed on each cell as a truncation of certain
		faces, but in such a way that the truncations of adjacent cells are compatible.
		As it turns out, this resulting complex is isomorphic to the cactus complex
		defined in \cite{nekrashevych14} and is homeomorphic to a subspace 
		of the dual braid complex described in \cite{brady10}. 
		The connections between these three complexes will be made 
		explicit in a future article.
	\end{rem}
	
	\section*{Acknowledgements}

	I am very grateful to Justin Lanier and Jon McCammond for many helpful conversations.
	
	\bibliographystyle{amsalpha}
	\bibliography{polyhedral-tree-complex-v2}

\providecommand{\bysame}{\leavevmode\hbox to3em{\hrulefill}\thinspace}
\providecommand{\MR}{\relax\ifhmode\unskip\space\fi MR }
\providecommand{\MRhref}[2]{%
  \href{http://www.ams.org/mathscinet-getitem?mr=#1}{#2}
}
\providecommand{\href}[2]{#2}
\begin{thebibliography}{AHBHY18}

\bibitem[AHBHY18]{arkani-hamed}
Nima Arkani-Hamed, Yuntao Bai, Song He, and Gongwang Yan, \emph{Scattering forms and the positive geometry of kinematics, color and the worldsheet}, J. High Energ. Phys \textbf{2018} (2018), no.~96.

\bibitem[Aur14]{auroux14}
Denis Auroux, \emph{A beginner's introduction to {F}ukaya categories}, Contact and symplectic topology, Bolyai Soc. Math. Stud., vol.~26, J\'{a}nos Bolyai Math. Soc., Budapest, 2014, pp.~85--136. \MR{3220941}

\bibitem[BHV01]{billera01}
Louis~J. Billera, Susan~P. Holmes, and Karen Vogtmann, \emph{Geometry of the space of phylogenetic trees}, Adv. in Appl. Math. \textbf{27} (2001), no.~4, 733--767.

\bibitem[BLMW22]{belk}
James Belk, Justin Lanier, Dan Margalit, and Rebecca~R. Winarski, \emph{Recognizing topological polynomials by lifting trees}, Duke Math. J. (2022).

\bibitem[BM10]{brady10}
Tom Brady and Jon McCammond, \emph{Braids, posets and orthoschemes}, Algebr. Geom. Topol. \textbf{10} (2010), no.~4, 2277--2314. \MR{2745672}

\bibitem[Bra01]{brady01}
Thomas Brady, \emph{A partial order on the symmetric group and new {$K(\pi,1)$}'s for the braid groups}, Adv. Math. \textbf{161} (2001), no.~1, 20--40. \MR{1857934}

\bibitem[BT94]{bott94}
Raoul Bott and Clifford Taubes, \emph{On the self-linking of knots}, vol.~35, 1994, Topology and physics, pp.~5247--5287. \MR{1295465}

\bibitem[CV86]{culler86}
Marc Culler and Karen Vogtmann, \emph{Moduli of graphs and automorphisms of free groups}, Invent. Math. \textbf{84} (1986), no.~1, 91--119. \MR{830040}

\bibitem[Dev99]{devadoss98}
Satyan~L. Devadoss, \emph{Tessellations of moduli spaces and the mosaic operad}, Homotopy invariant algebraic structures ({B}altimore, {MD}, 1998), Contemp. Math., vol. 239, Amer. Math. Soc., Providence, RI, 1999, pp.~91--114. \MR{1718078}

\bibitem[Dev03]{devadoss03}
\bysame, \emph{A space of cyclohedra}, Discrete Comput. Geom. \textbf{29} (2003), no.~1, 61--75.

\bibitem[DH84]{douady-hubbard-2}
A.~Douady and J.~H. Hubbard, \emph{\'{E}tude dynamique des polyn\^{o}mes complexes. {P}artie {I}}, Publications Math\'{e}matiques d'Orsay [Mathematical Publications of Orsay], vol.~84, Universit\'{e} de Paris-Sud, D\'{e}partement de Math\'{e}matiques, Orsay, 1984. \MR{762431}

\bibitem[DH85]{douady-hubbard-1}
\bysame, \emph{\'{E}tude dynamique des polyn\^{o}mes complexes. {P}artie {II}}, Publications Math\'{e}matiques d'Orsay [Mathematical Publications of Orsay], vol.~85, Universit\'{e} de Paris-Sud, D\'{e}partement de Math\'{e}matiques, Orsay, 1985, With the collaboration of P. Lavaurs, Tan Lei and P. Sentenac. \MR{812271}

\bibitem[FR07]{fomin07}
Sergey Fomin and Nathan Reading, \emph{Root systems and generalized associahedra}, Geometric combinatorics, IAS/Park City Math. Ser., vol.~13, Amer. Math. Soc., Providence, RI, 2007, pp.~63--131. \MR{2383126}

\bibitem[Mar99]{markl99}
Martin Markl, \emph{Simplex, associahedron, and cyclohedron}, Higher homotopy structures in topology and mathematical physics ({P}oughkeepsie, {NY}, 1996), Contemp. Math., vol. 227, Amer. Math. Soc., Providence, RI, 1999, pp.~235--265. \MR{1665469}

\bibitem[McC]{hypertrees}
Jon McCammond, \emph{Noncrossing hypertrees}.

\bibitem[Nek14]{nekrashevych14}
Volodymyr Nekrashevych, \emph{Combinatorial models of expanding dynamical systems}, Ergodic Theory Dynam. Systems \textbf{34} (2014), no.~3, 938--985. \MR{3199801}

\bibitem[Pen96]{penner96}
R.~C. Penner, \emph{The simplicial compactification of {R}iemann's moduli space}, Topology and {T}eichm\"{u}ller spaces ({K}atinkulta, 1995), World Sci. Publ., River Edge, NJ, 1996, pp.~237--252. \MR{1659667}

\bibitem[Sim03]{simion03}
Rodica Simion, \emph{A type-{B} associahedron}, vol.~30, 2003, Formal power series and algebraic combinatorics (Scottsdale, AZ, 2001), pp.~2--25.

\bibitem[Sta12]{stasheff12}
Jim Stasheff, \emph{How {I} `met' {D}ov {T}amari}, Associahedra, {T}amari lattices and related structures, Prog. Math. Phys., vol. 299, Birkh\"{a}user/Springer, Basel, 2012, pp.~45--63. \MR{3221533}

\bibitem[Wac07]{wachs07}
Michelle~L. Wachs, \emph{Poset topology: tools and applications}, Geometric combinatorics, IAS/Park City Math. Ser., vol.~13, Amer. Math. Soc., Providence, RI, 2007, pp.~497--615.

\bibitem[Zie95]{ziegler95}
G\"{u}nter~M. Ziegler, \emph{Lectures on polytopes}, Graduate Texts in Mathematics, vol. 152, Springer-Verlag, New York, 1995. \MR{1311028}

\end{thebibliography}

\end{document}